\documentclass[reqno]{amsart}
\usepackage{amsmath,amsfonts,amssymb,amscd,multicol}
\usepackage{xypic}
\usepackage[normalem]{ulem}
\usepackage{verbatim,color}

\usepackage{fullpage}

\newcommand{\on}{\operatorname}

\newcommand{\kk}{\Bbbk}

\newcommand{\bpr}{\begin{proof}}
\newcommand{\epr}{\end{proof}}

\newcommand{\mc}{\mathcal}
\newcommand{\mf}{\mathfrak}

\newcommand{\mb}{\mathbb}

\newcommand{\wt}{\widetilde}

\newcommand{\cProj}{\operatorname{Proj}}

\newcommand{\rcatmod}{\operatorname{mod-}}
\newcommand{\rcatMod}{\operatorname{Mod-}}

\newcommand{\rQch}{\operatorname{Qcoh}}

\newcommand{\rGr}{\operatorname{Gr-}\hskip -2pt}
\newcommand{\rQgr}{\operatorname{Qgr-}\hskip -2pt}

\newcommand{\rTors}{\operatorname{Tors-}\hskip -2pt}

\newcommand{\dirlim}{\underrightarrow{\lim}}

\newcommand{\wh}{\widehat}

\newcommand{\Qcoh}{\operatorname{Qcoh}}

\newcommand{\beq}{\begin{equation}}
\newcommand{\eeq}{\end{equation}}

\newcommand{\Hom}{{\rm Hom}}

\newcommand{\End}{{\rm End}}

\newcommand{\Ext}{{\rm Ext}}

\numberwithin{equation}{section}
 \theoremstyle{plain}
\newtheorem{theorem}[equation]{Theorem}

\newtheorem{corollary}[equation]{Corollary}
\newtheorem{proposition}[equation]{Proposition}

\theoremstyle{definition}

\newtheorem{definition}[equation]{Definition}

\newtheorem{hypothesis}[equation]{Hypothesis}
\newtheorem{standing-hypothesis}[equation]{Standing Hypothesis}
\newtheorem{example}[equation]{Example}

\begin{document}

\title{Closed subcategories of quotient categories}

\author{Daniel Rogalski}
\address{University of California, San Diego\\ Department of Mathematics\\ 9500 Gilman Dr. \#0112 \\
La Jolla, CA 92093-0112\\ USA}
\email{drogalski@ucsd.edu}
\date{\today}
\keywords{noncommutative geometry, noncommutative projective scheme, Grothendieck category, quotient category, quasi-scheme, weakly closed subcategory, closed subcategory}
\subjclass[2010]{
Primary:
18E10,  	
18E35.   
Secondary: 
14A22.   
}

\begin{abstract}
We study the spectrum of closed subcategories in a quasi-scheme, i.e. a Grothendieck category $X$.  The closed subcategories are the direct analogs of closed subschemes in the commutative case, in the sense that when $X$ is the category of quasi-coherent sheaves on a quasi-projective scheme $S$, then the closed subschemes of $S$ correspond bijectively to the closed subcategories of $X$.  Many interesting quasi-schemes, such as the noncommutative projective scheme $\rQgr B = \rGr B/\rTors B$ associated to a graded algebra $B$, arise as quotient categories of simpler abelian categories.  In this paper we will show how to describe the closed subcategories of any quotient category $X/Y$ in terms of closed subcategories of $X$ with special properties, when $X$ is a category with a set of compact projective generators.
\end{abstract}

\maketitle

\section{Introduction}

In commutative algebraic geometry, the closed subschemes $Z \subseteq S$ play an important role in the study of a scheme $S$.  In this paper we study one possible generalization of this idea to noncommutative geometry, following work of Rosenberg, Smith, Van den Bergh, Kanda, and others.  We take as the basic object of noncommutative geometry a Grothendieck category $X$, that is, an abelian category with exact direct limits and a generator.  Strictly speaking, this generalizes not a commutative scheme $S$ but rather its category of quasi-coherent sheaves $\rQch S$.  When taking this point of view, $X$ has also been called a \emph{quasi-scheme} (Rosenberg, Van den Bergh) or a \emph{space} (Smith).  

Let $X$ be a Grothendieck category.  A full subcategory $Z$ of $X$ is \emph{weakly closed} if it is closed under subquotients and direct sums.  The weakly closed subcategory $Z$ is \emph{closed} if it is also closed under products.

Our main goal in this paper is to further develop the theory of closed subcategories, although we will also prove results about weakly closed subcategories along the way.  Closed subcategories are the most direct analogues of closed subschemes in the commutative case.  Indeed, if $S$ is a quasi-projective scheme over a field, then the closed subschemes of $S$ are in bijection with the closed subcategories of $\rQch S$ \cite[Theorem 4.1]{Smi02}.  As further evidence that the definition of closed is reasonable, if $R$ is any unital ring we can consider the category $\rcatMod R$ of right $R$-modules as a noncommutative affine quasi-scheme.  In this case, the closed subcategories of $\rcatMod R$ are precisely those of the form $\rcatMod (R/I)$ for ideals $I$ of $R$, as proved by Rosenberg \cite[Proposition 6.4.1]{Ros95}.  Letzter \cite{Let07} also noted that in any quasi-scheme $X$ with a single generator $O$, the closed subcategories of $X$ correspond bijectively to certain special subobjects of $O$ which by analogy he called ideals, though he did not give a characterization of these subobjects.

Many of the major developments in this kind of noncommutative geometry have concerned analogs of projective schemes.  Suppose that $B = \bigoplus_{n = 0}^{\infty} B_n$ is an $\mb{N}$-graded $\kk$-algebra which is connected $(B_0 = \kk)$ and finitely generated as an algebra by $B_1$.  Let $\rGr B$ be the category of $\mb{Z}$-graded right $B$-modules, and let $\rTors B$ be the subcategory of $\rGr B$ consisting of direct limits of modules which are finite dimensional over $\kk$.  Using a construction of Gabriel, one forms the quotient category $\rQgr B = \rGr B/\rTors B$.  The category $\rQgr B$ was studied by Artin and Zhang \cite{AZ1} and is called a  \emph{noncommutative projective scheme}.  The same idea appeared in work of Verevkin \cite{Ver92}.  In case $B$ is commutative, $\rQgr B \simeq \rQch S$ for the projective scheme $S = \cProj B$, and as already noted the closed subcategories of $\rQch S$ simply correspond to the closed subschemes of $S$.  Moreover, it is standard that these closed subschemes are in bijection with the graded ideals $I$ of $B$ such that $I$ is \emph{saturated}, i.e. $B/I$ has no nonzero submodules in $\rTors B$.  Here, $I$ corresponds to the closed subscheme $\cProj (B/I)$ of $S$, and $\rQgr  \, (B/I)$ is the corresponding closed subcategory of $\rQgr B$.

By contrast, closed subcategories of noncommutative projective schemes $X = \rQgr B$ are not so easy to describe when $B$ is a noncommutative $\mb{N}$-graded $\kk$-algebra.  If $B$ is noetherian and $I$ is a graded ideal, Smith proved that $Z = \rQgr \, (B/I)$ is indeed closed in $\rQgr B$, but the proof is non-trivial \cite[Theorem 1.2]{Smi16}.  Moreover, 
in general $X$ may have many closed subcategories that are not of this form.  For example, if $I$ is a graded right ideal of $B$ such that $\dim_{\kk} (B/I)_n = 1$ for all $n \geq 1$, then $B/I$ is called a \emph{point module}.  Under mild additional hypotheses, the image of the point module in $X = \rQgr B$ is a simple object $p$ and the semisimple subcategory consisting of all direct sums of copies of $p$ forms a closed subcategory of $X$ \cite[Proposition 5.8]{Smi02}, but $I$ is generally not a $2$-sided ideal.  

In this paper we develop a framework which will allow one to better understand the spectrum of closed subcategories of a noncommutative projective scheme $\rQgr B$.  The problem of describing the closed subcategories of $\rQgr B$ naturally breaks into two pieces:  (1) understand the closed subcategories of the category $\rGr B$ of $\mb{Z}$-graded modules over a graded ring $B$; and (2) understand how the closed subcategories of $\rGr B$ and those of its quotient category $\rQgr B = \rGr B/\rTors B$ are related.  We can formulate both of these parts as instances of very general questions.  First, note that a module category $\rcatMod R$ has a single projective generator $R$, while if $B$ is a $\mb{Z}$-graded ring then $\rGr B$ is generated by the family of projective generators $\{ B(n) | n \in \mb{Z} \}$, where $B(n)$ is the graded module with graded pieces $B(n)_i = B_{n+i}$.  So for (1), we can ask if Rosenberg's description of closed subcategories of $\rcatMod R$ in terms of ideals of $R$ has a natural extension to categories with a set of projective generators.  For (2), we can seek to understand the relation between the closed subcategories of a general Grothendieck category $X$ and the closed subcategories of a quotient $X/Y$.  Because these turn out to be non-trivial and interesting problems in their own right, with applications to many examples beyond noncommutative projective schemes, we concentrate on the solution of these general problems here.  We will then apply the theory we develop to the special case of noncommutative projective schemes in a sequel paper \cite{Rog25}.

Problem (1) described in the previous paragraph turns out to have a relatively straightforward solution.  An object $M$ of a Grothendieck category $X$ is called \emph{compact} if the functor $\Hom_X(M, -)$ commutes with arbitrary direct sums.  Compactness is a natural generalization of the notion of finite generation in a module category.  Our first result shows how to describe the closed subcategories of any category with a set of compact projective generators.
\begin{theorem} [Theorem~\ref{thm:weakly-closed}, Corollary~\ref{cor:idealinprojcase}]
\label{thm:main1}
Let $X$ be a Grothendieck category, with small set of compact projective generators $S = \{ O_{\alpha} \}_{\alpha \in A}$.  A collection $\{ J_{\alpha} \}_{\alpha \in A}$, where $J_{\alpha} \subseteq O_{\alpha}$ for all $\alpha$, is called an \emph{ideal} in the set of generators $S$ if whenever $f: O_{\alpha} \to O_{\beta}$ is a morphism in $X$, then $f(J_{\alpha}) \subseteq J_{\beta}$.

Every closed subcategory $Z$ of $X$ is generated by the objects $\{ O_{\alpha}/J_{\alpha} \}$ for a unique ideal $\{ J_{\alpha} \}$ in the set of generators $S$.  Thus, there is a bijective correspondence between closed subcategories in $X$ and ideals in the set of generators $S$. 
\end{theorem}
It is easy to see that when $X = \rcatMod R$ and $S = \{ R \}$, an ideal in the set of generators $S$ is precisely an ideal $J$ of $R$.  So Theorem~\ref{thm:main1} directly generalizes Rosenberg's result. In the course of the proof of Theorem~\ref{thm:main1}, we also give an analogous result about the weakly closed subcategories of $X$, showing that they correspond to choices of filters of subobjects inside each generator with certain compatibility conditions (see Theorem~\ref{thm:weakly-closed}). 

The category $X$ in Theorem~\ref{thm:main1} can also be reinterpreted as a category of modules over the $A$-ring $\displaystyle\bigoplus_{\alpha, \beta \in A} \Hom_X(O_{\beta}, O_{\alpha})$, in which case the notion of ideal in a set of generators becomes a literal ideal in the $A$-ring.  This point of view is more important for \cite{Rog25}, so we defer the discussion to that paper.

Next, we describe our main result about closed subcategories of quotients, which solves general problem (2) stated above.  Given a Grothendieck category $X$, recall that a subcategory $Y$ is \emph{localizing} if it is weakly closed and also closed under extensions.  In this case there is a quotient category $X/Y$ together with an exact quotient functor $\pi: X \to X/Y$ which has a right adjoint $\omega: X/Y \to X$.  An object $M$ in $X$ has a unique largest subobject $\tau(M)$ in $Y$, called the $Y$-torsion submodule of $M$.  The object $M$ is called $Y$-torsionfree if $\tau(M) = 0$.
\begin{theorem}[Theorem~\ref{thm:quotientclosed}]
\label{thm:main2}
Let $X$ be a Grothendieck category such that products are exact (Grothendieck's (AB4*) condition).  Let $Y$ be a localizing subcategory of $X$ and consider the quotient category $X/Y$.  A closed subcategory $Z$ of $X$ is called \emph{$Y$-closed} if (i) $Z$ is \emph{$Y$-torsionfree generated}, that is the full subcategory generated by the $Y$-torsionfree objects in $Z$ is all of $Z$; and (ii) $Z$ is \emph{$Y$-essentially stable}: if $M \in Z$ then $\omega \pi(M) \in Z$.
\begin{enumerate}
    \item If $Z$ is a closed subcategory of $X$ which is $Y$-essentially stable, then $\pi(Z) = \{ \pi(M) | M \in Z \}$ is a closed subcategory of $X/Y$. 
    \item If $\mc{Z}$ is a closed subcategory of $X/Y$, then the subcategory $Z$ of $X$ generated by all $Y$-torsionfree $M \in X$ with $\pi(M) \in \mc{Z}$ is a $Y$-closed subcategory of $X$.
    \item There is a bijective correspondence between $Y$-closed subcategories $Z$ of $X$ and closed subcategories $\mc{Z}$ of $X/Y$ given by the inverse operations in (1) and (2).  
\end{enumerate}
\end{theorem}
We first prove Theorem~\ref{thm:main2} in the context of weakly closed subcategories (in this case the (AB4*) condition is not needed, nor is the hypothesis in (1) that $Z$ is $Y$-essentially stable).  Then we show it specializes to a bijection on closed subcategories.  We note that Kanda has given a different bijection of the weakly closed subcategories of a quotient category $X/Y$ with certain weakly closed subcategories of $X$ \cite[Proposition 4.13]{Kan15}, which does not restrict to a bijection on the closed subcategories.  The bijection we give in Theorem~\ref{thm:main2} is anticipated by the work of Smith in \cite[Proposition 7.10]{Smi02}, though with a different point of view; we explain the connection at the end of Section~\ref{sec:operations} below.  

As illustrations of our theorems, we give a number of examples to demonstrate some of the subtleties of closed subcategories that do not occur for weakly closed ones.  These can be found already in the following simple setting:  Let $R$ be a noetherian ring with ideal $I$, and let 
$Y$ be the localizing subcategory of $X = \rcatMod R$ consisting of modules supported along the closed subcategory $\rcatMod(R/I)$, namely 
\[
Y = \{ M \in \rcatMod R \, | \, \text{for all}\ m \in M\ \text{there exists}\ n \geq 1\ \text{such that}\ m I^n = 0 \}.
\]
Let $K$ be another ideal, so $Z = \rcatMod (R/K)$ is closed in $X$.

We give such an example (Example~\ref{ex:pinotclosed}) where $\pi(Z)$ is not closed in $\mc{X} = X/Y$, only weakly closed, showing that the $Y$-essentially stable hypothesis in Theorem~\ref{thm:main2}(1) is necessary in general.  We also show that there is an example of this type with a further quotient map $\wt{\pi}: \mc{X} \to \mc{X}/\mc{Y}$ such that Theorem~\ref{thm:main2}(2) fails for this quotient (Example~\ref{ex:weird}), showing the necessity of the (AB4*) assumption.

Smith defined and studied the natural intersection and union operations on weakly closed and closed subcategories of an arbitrary quasi-scheme $X$ in \cite{Smi02}:  if $W$ and $Z$ are weakly closed in $X$, then $W \cap Z$ consists of the objects in both $W$ and $Z$, while $W \cup Z$ is the smallest weakly closed subcategory containing $W$ and $Z$.  It is obvious that if $Z$ and $W$ 
are closed in $X$ then so is $Z \cap W$, but notably no claim is made in \cite{Smi02} that $Z \cup W$ is closed in general, though Smith notes this is true when $X = \rcatMod R$ for a ring $R$ \cite[Remark 3, p. 2138]{Smi02}.  We produce a ring $R$, an ideal $I$ and ideals $K_1, K_2$ 
such that if $Y$ is the localizing subcategory associated to $I$ as above, and $Z_1 = \rcatMod (R/K_1)$, $Z_2 = \rcatMod (R/K_2)$, then $\pi(Z_1)$ and and $\pi(Z_2)$ are closed in $X/Y$ but $\pi(Z_1) \cup \pi(Z_2)$ is not (Example~\ref{ex:badunion}).  The key point is that $Z_1$ and $Z_2$ are $Y$-essentially stable in $X$, but $Z_1 \cup Z_2$ is not.

We also provide a simple counterexample to a question of Smith \cite[Remark 5, p. 2139]{Smi02} asking whether the operations of union and intersection for weakly closed subcategories distribute over each other (Example~\ref{ex:notdist}).

While we have chosen to focus on closed subcategories in this paper, we are not claiming that this is the correct or most important kind of spectrum of a quasi-scheme to study.  There are also other important spectra which give different kinds of information, for example the \emph{atom spectrum} studied by Kanda in a  series of papers beginning with \cite{Kan12}. Closed subcategories do seem to be the best formal analogues of closed subschemes.  They are essential ingredients, for example, in Van den Bergh's machinery of noncommutative blowing up \cite{VdB01}.  So we believe that closed subcategories are important and worth understanding in more detail.   On the other hand, closed subcategories definitely have their limitations:  just as noncommutative rings tend to have relatively few ideals compared to their commutative counterparts, interesting examples such as noncommutative projective schemes often have relatively few closed subcategories in general---too few to fully represent the underlying geometry.  For example, Smith and Zhang have studied some special weakly closed but not closed subcategories which can be thought of as nice noncommutative curves in noncommutative projective planes \cite{SmZh98}.  However, general weakly closed subcategories in quasi-schemes are so plentiful that they do not all seem to be geometrically meaningful.  It would be interesting to determine a general class of subcategories with especially good properties that is intermediate between weakly closed and closed subcategories.

\subsection*{Acknowledgments}
We thank Ryo Kanda and Paul Smith for helpful conversations.

\section{Categorical background}
\label{sec:background}

In this paper we use standard conventions for avoiding set-theoretic problems in our study of categories.  We fix a Grothendieck universe $U$ and call the sets in $U$ \emph{small}.  The objects in a category $C$ may not form a small set, but $\Hom_C(M,N)$ is assumed small for any objects $M, N \in C$.  All index sets for limits and colimits, in particular for coproducts and products, are assumed to be small sets.  We prefer to refer to coproducts in an abelian category as direct sums.

Suppose that $X$ is an abelian category with arbitrary small direct sums (Grothendieck's condition (AB3)).
Recall that an object $O$ in $X$ is a \emph{generator} if whenever $0 \neq f \in \Hom_X(Y, Z)$ 
then there is a morphism $g: O \to Y$ such that $0 \neq f \circ g \in \Hom(O, Z)$.  It is standard that $O \in X$ is a generator if and only if every object of $X$ is an epimorphic image of a direct sum of copies of $O$.  Similarly, a small-indexed set of objects  $S = \{ O_{\alpha} | \alpha \in A \}$ is said to be a \emph{set of generators} if $O = \bigoplus_{\alpha \in A} O_{\alpha}$ is a generator; equivalently, if every object in $X$ is an epimorphic image of a direct sum of objects taken from $S$.  In any (AB3) abelian category $X$ with a small set of generators, the set of subobjects of any given object in $X$ is small---that is, $X$ is \emph{well-powered}, or \emph{locally small} \cite[Proposition 1.2.2, Proposition 3.4]{Pop73}.

\begin{definition}
Let $X$ be an abelian category.  We say $X$ is a \emph{Grothendieck category} if $X$ has a generator $O$ and 
satisfies Grothendieck's condition (AB5):  (small) filtered colimits are exact.
\end{definition}

A Grothendieck category automatically has a number of other good properties.  The (AB5) condition implies (AB4), that direct sums are exact \cite[Corollary 2.8.9]{Pop73}.  A Grothendieck category automatically has (small) products (condition (AB3*)) \cite[Corollary 7.10]{Pop73}, but products need not be exact.  Every object in a Grothendieck category has an injective hull \cite[Theorem 3.10.10]{Pop73}, so in particular there are enough injective objects.  

It is often more convenient to work with a set of generators each of which has a reasonable size, rather than one big generator.  For example, one might want to assume each generator is noetherian.  We primarily need the following weaker condition in this paper.  An object $M$ in an abelian category $X$ is called \emph{compact} (\cite{Pop73} uses the word \emph{small}) if whenever $\phi: M \to \bigoplus_{\alpha \in A} N_{\alpha}$ is a morphism in $X$, there is some finite set of indices $\alpha_1, \dots, \alpha_k$ such that $\phi(M) \subseteq \bigoplus_{i=1}^k N_{\alpha_i}$.  Alternatively, $M$ is compact if and only if the functor $\Hom_X(M, -)$ commutes with arbitrary direct sums \cite[Proposition 3.5.1]{Pop73}.

As described in the introduction, we are interested in understanding subcategories of a Grothendieck category $X$ with 
special properties.  Throughout this paper we are almost exclusively interested in full subcategories of a given category, and so we describe them based on the properties of the objects they contain.  Unless otherwise stated the reader should assume that all subcategories are full.
\begin{definition}
Let $Z$ be a full subcategory of a Grothendieck category $X$.   
\begin{enumerate}
\item 
We say that $Z$ is \emph{pre-localizing} or \emph{weakly closed} if (i) for any $M \in Z$, all subobjects and quotient objects of $M$ are in $Z$ (informally, we say that $Z$ is \emph{closed under subquotients}), and (ii) for any (small) set $\{M_{\alpha} \}$ with each $M_{\alpha} \in Z$, then $\bigoplus_{\alpha} M_{\alpha} \in Z$ (i.e. $Z$ is \emph{closed under direct sums}).
We will prefer the term weakly closed in this paper.
\item 
We say that $Z$ is \emph{localizing} if $Z$ is weakly closed and if whenever $0 \to M \to N \to P \to 0$ is a short exact sequence in $X$ with $M, P \in Z$, then $N \in Z$ (i.e. $Z$ is \emph{closed under extensions}).
\item We say that $Z$ is \emph{closed} if it is weakly closed and for any (small) set $\{M_{\alpha} \}$ with each $M_{\alpha} \in Z$, then $\prod_{\alpha} M_{\alpha} \in Z$ (i.e. $Z$ is \emph{closed under products}).
 \end{enumerate}
\end{definition}
Note that to check a subcategory $Z$ is closed we only need to check that $Z$ is closed under subquotients and products, and do not need to check separately that $Z$ is closed under direct sums.  This is because in a Grothendieck category 
the natural morphism $\bigoplus M_{\alpha} \to \prod M_{\alpha}$ is a monomorphism \cite[Corollary 8.10]{Pop73}. 

We should mention that weakly closed and closed subcategories also have natural functorial descriptions, which are sometimes taken as the definitions, as in \cite[Definition 2.4]{Smi02}.  If $Z$ is a full subcategory of $X$, and $i: Z \to X$ is the inclusion functor, then $Z$ is weakly closed if and only if $i$ has a right adjoint $i^!: X \to Z$. Explicitly, in this case $i^!(M)$ is the unique largest subobject of $M$ in $Z$.  Similarly, $Z$ is closed if and only if $i$ has a left adjoint $i^*:X \to Z$, and in this case $i^*(M)$ is the unique largest factor object of $M$ in $Z$.  

Next, we review the facts about quotient categories that we will need.  Let $X$ be a Grothendieck category.
Let $Y$ be a localizing subcategory of $X$. Because $Y$ is weakly closed, every object $M \in X$ has a unique largest subobject in $Y$, called its \emph{$Y$-torsion subobject} $\tau(M)$ \cite[Proposition 4.4.5]{Pop73}.  We say that the object $M$ is $Y$-\emph{torsionfree} if $\tau(M) = 0$ and $Y$-\emph{torsion} if $\tau(M) = M$; that is, if $M \in Y$.   If $\{ M_{\alpha} \}$ is a collection of $Y$-torsionfree objects and $N$ is $Y$-torsion, then any morphism $N \to M_{\alpha}$ is $0$, so any morphism $N \to \prod_{\alpha} M_{\alpha}$ is $0$.  It follows that $\prod_{\alpha} M_{\alpha}$ is also $Y$-torsionfree, and the same holds for its subobject $\bigoplus_{\alpha} M_{\alpha}$.

Following Gabriel, we define the \emph{quotient category} $X/Y$.  This category has the same objects as $X$, and there is an exact quotient functor $\pi: X \to X/Y$ which is the identity map on objects.  The morphisms in $X/Y$ are given by $\Hom_{X/Y}(\pi M, \pi N) = \dirlim\ \Hom_X( M', N/N')$, the direct limit running over all subobjects $M' \subseteq M$ with $M/M' \in Y$ and all subobjects $N' \subseteq N$ with $N' \in Y$ \cite[Theorem 4.3.3]{Pop73}.  The map $\pi$ has the following universal property:  If $F: X \to W$ is an exact functor between abelian categories such that $F(N) = 0$ for all $N \in Y$, then there is a unique functor $\overline{F}: X/Y \to W$ such that $\overline{F} \circ \pi = F$ \cite[Corollary 4.3.11]{Pop73}. Because we assume that $X$ is Grothendieck, $X/Y$ is also Grothendieck \cite[Corollary 4.6.2]{Pop73}.  Moreover if $S = \{ O_{\alpha} | \alpha \in A \}$ generates $X$, then $\pi(S) = \{ \pi(O_{\alpha}) | \alpha \in A \}$ generates $X/Y$ \cite[Lemma 4.4.8]{Pop73}.

The quotient functor $\pi$ has a right adjoint $\omega: X/Y \to X$ called the \emph{section functor}. The counit $v: \pi \omega \to \on{id}_{X/Y}$ of the adjunction is an isomorphism \cite[Proposition 4.4.3]{Pop73}.  The unit $u: \on{id}_X \to \omega \pi$ has the following explicit description.  Given $M \in X$, let $M' = M/\tau(M)$, and let $E(M')$ be the injective hull of $M'$ in $X$.  Then $\omega \pi(M) = N$, where $N/M' = \tau(E(M')/M')$, and $u_M: M \to \omega \pi(M)$ is the natural epimorphism $M \to M'$ followed by the inclusion $M' \to N$.   In other words, $\omega \pi$ first mods out by the $Y$-torsion submodule, and then takes the unique largest essential extension by an object in $Y$ \cite[Lemma 4.4.1(2), Lemma 4.4.6(2)]{Pop73}.  

The functor $\pi$ preserves coproducts (being a left adjoint), so that given a collection $M_{\alpha}$ of objects in $X$, the direct sum $\bigoplus \pi(M_{\alpha})$ in $X/Y$ is simply equal to $\pi(\bigoplus M_{\alpha})$.   The functor $\pi$ does not preserve products in general, but $\omega$ does (being a right adjoint).  It follows that given a collection $\mc{M}_{\alpha}$ of objects in $X/Y$, the product $\prod \mc{M}_{\alpha}$ can be described as $\pi( \prod \omega(\mc{M}_{\alpha}))$.

\section{Special subcategories and systems of filters in a set of generators}
\label{sec:generators}

Consider $X = \rcatMod R$, the category of right modules over a ring $R$.  In this case the weakly closed, closed, and localizing subcategories of $X$ can be described in terms of certain families of right ideals of $R$ called filters.  This idea essentially goes back to Gabriel's seminal paper \cite[Section V.2]{Gab62}. See also the work of Kanda \cite{Kan15}, where the idea is generalized to categories of quasicoherent sheaves over any locally noetherian scheme.

Note that a module category $\rcatMod R$ has the single compact projective generator $R$.   In this section we will assume the following more general hypothesis on $X$:
\begin{hypothesis}
\label{hyp-standing}
\label{hyp:standing}
Let $X$ be a Grothendieck category with a small set of compact projective generators $S = \{ O_{\alpha} | \alpha \in A \}$.
\end{hypothesis}
The results of this section will give a description of the weakly closed, closed, and localizing subcategories in the setting of Hypothesis~\ref{hyp-standing}, in terms of filters inside each of the compact projective generators $\{ O_{\alpha} \}$.

\begin{definition}
Let $O \in X$.  A \emph{filter} inside $O$ is a collection $\mc{F}$ of subobjects of $O$ such that (i) $O \in \mc{F}$; (ii) $\mc{F}$ is closed under specialization, i.e. if $I \subseteq J$ are subobjects of $O$ and $I \in \mc{F}$, then $J \in \mc{F}$; and (iii) $\mc{F}$ is closed under (finite) intersection, i.e. if $I, J \in \mc{F}$ then $I \cap J \in \mc{F}$.  A filter $\mc{F}$ is \emph{principal} if there is $I \in \mc{F}$ such that 
$\mc{F} = \{ J \subseteq O | I \subseteq J \}$, i.e. $\mc{F}$ consists of all specializations of some fixed subobject $I$ of $O$.  
\end{definition}
Note that in the presence of property (ii), property (i) is equivalent to assuming the filter is nonempty.
Clearly a filter $\mc{F}$ is principal if and only if it is closed under arbitrary intersections.

\begin{definition}
Assume Hypothesis~\ref{hyp:standing}.  A \emph{filter system} in the given set of generators is a choice of filter $\mc{F}_{\alpha}$ of $O_{\alpha}$ for each $\alpha$, such that 
\begin{equation}
\text{for any morphism}\ f: O_{\alpha} \to O_{\beta},\ \text{if}\ J \in \mc{F}_{\beta}\ \text{then}\ f^{-1}(J) \in \mc{F}_{\alpha}. \tag{$\dagger$}
\end{equation}
A filter system $\{ \mc{F}_{\alpha} \}$ is \emph{principal} if every filter $\mc{F}_{\alpha}$ is principal.  We also say that a filter system is \emph{Gabriel} if given any epimorphism $g:  \bigoplus_i O_{\gamma_i}/I_{\gamma_i} \to J/K$ in $X$, where $K \subseteq J \subseteq O_{\beta}$ with each $I_{\gamma_i} \in \mc{F}_{\gamma_i}$ and $J \in \mc{F}_{\beta}$, then $K \in \mc{F}_{\beta}$.
\end{definition}

The condition $(\dagger)$ in the definition of filter system has the following alternative expression: 
\begin{equation}
\text{for any morphism}\ f: O_{\alpha} \to O_{\beta}/J\ \text{with}\ J \in \mc{F}_{\beta},\ \ker(f) \in \mc{F}_{\alpha}. \tag{$\dagger'$}
\end{equation}
The equivalence of $(\dagger)$ and $(\dagger')$ follows easily from the projectivity of the generators $O_{\beta}$.  
We freely use this alternative below.

\begin{theorem}
\label{thm:weakly-closed}
\label{prop:weakly-closed}
Assume Hypothesis~\ref{hyp-standing}.
\begin{enumerate}
    \item Given a filter system $\{ \mc{F}_{\alpha} \}$ in the set of generators $\{ O_{\alpha} \}$, the subcategory $Z$ of $X$ generated by all objects $\{ O_{\alpha}/I_{\alpha} \, | \, I_{\alpha} \in \mc{F}_{\alpha} \}$ is a weakly closed subcategory of $X$.
\item Given a weakly closed subcategory $Z$ of $X$, Let $\mc{F}_{\alpha}$ be the collection of all subobjects $I$ of $O_{\alpha}$ such that $O_{\alpha}/I \in Z$.  Then $\mc{F}_{\alpha}$ is a filter in $O_{\alpha}$ and $\{ \mc{F}_{\alpha} \}$ is a filter system in $\{ O_{\alpha} \}$.
\item The operations in (1) and (2) give a bijective correspondence between weakly closed subcategories $Z$ of $X$ and filter systems $\{ \mc{F}_{\alpha} \}$ in the set of generators $S = \{ O_{\alpha} \}$.  
\item In the correspondence of part (3), $Z$ is closed (resp. localizing) if and only the filter system $\{ \mc{F}_{\alpha} \}$ is principal (resp. Gabriel).
\end{enumerate}
  \end{theorem}
\begin{proof}
(1) Because $Z$ consists of all quotient objects of direct sums of the generators, and direct sums are exact, $Z$ is automatically closed under direct sums and quotient objects.  We just need to check that $Z$ is closed under subobjects.  For this, since $Z$ is closed under quotient objects, it is enough to check that given a direct sum $M = \bigoplus_k O_{\beta_k}/I_{\beta_k}$, where $I_{\beta_k} \in \mc{F}_{\beta_k}$ for each $\beta_k$, any subobject of $M$ is in $Z$.  Let $N$ be a subobject of $M$, and choose an epimorphism $f: \bigoplus_{\gamma_i} O_{\gamma_i} \to N$.  Let $f_i: O_{\gamma_i} \to N$ be the $i$th coordinate morphism of $f$.  Then $N = \sum_{i} f_i(O_{\gamma_i})$.  Since $Z$ is closed under direct sums and quotient objects, it is closed under sums of subobjects.  Thus it is enough to prove that each $f_i(O_{\gamma_i})$ is in $Z$.

We have reduced to showing that if $g: O_{\gamma} \to M = \bigoplus O_{\beta_k}/I_{\beta_k}$ is a morphism in $X$, then its image is in $Z$.  Because $O_{\gamma}$ is compact, we can replace $M$ with a finite direct sum and assume that 
$M = \bigoplus_{k=1}^n O_{\beta_k}/I_{\beta_k}$.  Now the image of $g$ is isomorphic to $O_{\gamma}/\ker g$.  Moreover, $\ker g = \bigcap_{k=1}^n \ker g_k$, where $g_k = \pi_k \circ g: O_{\gamma} \to O_{\beta_k}/I_{\beta_k}$, with $\pi_k$ the $k$th projection map.  Since $\mc{F}_{\gamma}$ is closed under finite intersection, we just need $\ker g_k \in \mc{F}_{\gamma}$ for each $k$.

With this further reduction, we now consider $h: O_{\gamma} \to O_{\beta}/I$ where $I \in \mc{F}_{\beta}$, and we need to prove that $\ker h \in \mc{F}_{\gamma}$.  This is precisely the alternative phrasing $(\dagger')$ of the defining property of a filter system.

(2) Let $\mc{F}_{\alpha}$ be the set of subobjects $I$ of $O_{\alpha}$ such that $O_{\alpha}/I \in Z$, where $Z$ is weakly closed.  Obviously $O_{\alpha} \in \mc{F}_{\alpha}$.  If $I \in \mc{F}_{\alpha}$ and $I \subseteq J$, then clearly $J \in \mc{F}_{\alpha}$ since $Z$ is closed under factor objects.  Since $O_{\alpha}/(I \cap J)$ embeds in $(O_{\alpha}/I) \oplus (O_{\alpha}/J)$, using that $Z$ is closed under direct sums and subobjects we see that $\mc{F}_{\alpha}$ is closed under finite intersection.  So $F_{\alpha}$ is a filter of $O_{\alpha}$.

Given a morphism $f: O_{\alpha} \to O_{\beta}$ with $I \in \mc{F}_{\beta}$, then $O_{\alpha}/(f^{-1}(I))$ is isomorphic to a subobject of $O_{\beta}/I \in Z$.  So $O_{\alpha}/(f^{-1}(I)) \in Z$ and thus $f^{-1}(I) \in \mc{F}_{\alpha}$, proving that $\{ \mc{F}_{\alpha} \}$ is a filter system.

(3) Let $Z$ be weakly closed and define the filter system $\{\mc{F}_{\alpha} \}$ as in (2), where $\mc{F}_{\alpha}$ consists of the subobjects $I$ of $O_{\alpha}$ such that $O_{\alpha}/I \in Z$.  Let $Z'$ be  the subcategory generated by all objects $O_{\alpha}/I$ with $I \in \mc{F}_{\alpha}$.  By definition $Z'$ is generated by objects in $Z$, so $Z' \subseteq Z$.  Conversely, suppose that $M \in Z$.  Choose some epimorphism $f: \bigoplus_i O_{\gamma_i} \to M$.  Each coordinate morphism $f_i: O_{\gamma_i} \to M$ has image in $Z$ and so $I_{\gamma_i} = \ker f_i \in \mc{F}_{\gamma_i}$.  Thus $f$ factors through to give an epimorphism $\bigoplus O_{\gamma_i}/I_{\gamma_i} \to M$,
which shows that $M \in Z'$.  Thus $Z \subseteq Z'$ and $Z = Z'$.

Conversely, if we start with a filter system $\{ \mc{F}_{\alpha} \}$, let $Z$ be generated by all objects $O_{\alpha}/I_{\alpha}$ with $I_{\alpha} \in \mc{F}_{\alpha}$.  Define a filter system $\{ \mc{F}'_{\alpha} \}$, where $\mc{F}'_{\alpha}$ consists of all subobjects $I$ of $O_{\alpha}$ such that 
$O_{\alpha}/I \in Z$.  It is obvious that $\mc{F}_{\alpha} \subseteq \mc{F}'_{\alpha}$ for each $\alpha$.  If $I \in \mc{F}'_{\alpha}$, then 
$O_{\alpha}/I \in Z$, so there is an epimorphism $h: \bigoplus_i O_{\gamma_i}/I_{\gamma_i} \to O_{\alpha}/I$ with all $I_{\gamma_i} \in \mc{F}_{\gamma_i}$.  Since $O_{\alpha}$ is projective, there is a map $\wh{\pi}: O_{\alpha} \to \bigoplus_i O_{\gamma_i}/I_{\gamma_i}$ such that $h \circ \wh{\pi}= \pi$, where $\pi: O_{\alpha} \to O_{\alpha}/I$ is the natural epimorphism. As $O_{\alpha}$ is compact, the image of $\wh{\pi}$ must be contained in finitely many summands.  This shows that we may replace $\bigoplus_i O_{\gamma_i}/I_{\gamma_i}$ with a finite sum $\bigoplus_{i=1}^k O_{\gamma_i}/I_{\gamma_i}$ and $h$ with its restriction 
$h: \bigoplus_{i=1}^k O_{\gamma_i}/I_{\gamma_i} \to O_{\alpha}/I$, which must still be an epimorphism.

For each $i$, let $\wh{\pi}_i = p_i \circ \wh{\pi} : O_{\alpha} \to O_{\gamma_i}/I_{\gamma_i}$, for the $i$th coordinate projection $p_i$.  By $(\dagger')$, $I_{\gamma_i} \in \mc{F}_{\gamma_i}$ implies $\ker \wh{\pi}_i \in \mc{F}_{\alpha}$.
Then $\ker \wh{\pi} = \bigcap_{i=1}^k \ker \wh{\pi}_i \in \mc{F}_{\alpha}$ since a filter is closed under finite intersection.
Finally, $\ker \wh{\pi} \subseteq \ker \pi = I$, so $I \in \mc{F}_{\alpha}$.  Thus $\mc{F}'_{\alpha} \subseteq \mc{F}_{\alpha}$ for all $\alpha$ as well, so $\mc{F}_{\alpha} = \mc{F}'_{\alpha}$.

(4)  Let $Z$ be a closed subcategory of $X$, and let $\{\mc{F}_{\alpha} \}$ be the corresponding filter system.  Let $\{ I_{\beta} \}$ be a family of objects in $\mc{F}_{\alpha}$.  Consider the 
map $f: O_{\alpha} \to \prod_{\beta} O_{\alpha}/I_{\beta}$ whose coordinate map $f_{\beta}: O_{\alpha} \to O_{\alpha}/I_{\beta}$ is the canonical epimorphism for each $\beta$.
Note that $\ker f = \bigcap_{\beta} I_{\beta}$.  Because $Z$ is closed under products and subobjects, 
$O_{\alpha}/(\ker f) \in Z$ and hence $\bigcap_{\beta} I_{\beta} \in \mc{F}_{\alpha}$.  Thus $\mc{F}_{\alpha}$ is closed under arbitrary intersections and so is a principal filter.  By definition $\{\mc{F}_{\alpha} \}$ is a principal filter system.

Conversely, suppose that $\{ \mc{F}_{\alpha} \}$ is a principal filter system.  Let $Z$ be the corresponding weakly closed subcategory of $X$.  
Let $P = \prod_{\beta} M_{\beta}$ be a product with each $M_{\beta} \in Z$.  Consider the image of a map $f: O_{\alpha} \to P$, which is isomorphic to $O_{\alpha}/\ker f$.  Here,  $\ker f = \bigcap_{\beta} \ker f_{\beta}$, where $f_{\beta} = \pi_{\beta} \circ f: O_{\alpha} \to M_{\beta}$ is the coordinate map.  Clearly $O_{\alpha}/(\ker f_{\beta}) \in Z$, since it is a subobject of $M_{\beta} \in Z$.  Then $\ker f_{\beta} \in \mc{F}_{\alpha}$, because by the bijective correspondence, $\mc{F}_{\alpha}$ contains all subobjects $I$ of $O_{\alpha}$ with $O_{\alpha}/I \in Z$.  Now $\mc{F}_{\alpha}$ is closed under arbitrary intersection 
and thus $\ker f \in \mc{F}_{\alpha}$.  Hence the image of $f$ is in $Z$.  Now $P$ is a sum of subobjects which are images of generators $O_{\alpha}$; since $Z$ is closed under sums of subobjects, $P \in Z$.  We conclude that $Z$ is closed under products, so $Z$ is a closed subcategory.

We have included the characterization of localizing subcategories in terms of Gabriel filter systems for completeness, but we will not need it below.  We leave the proof to the reader.  
\end{proof}

Suppose that $Z$ is a closed subcategory of $X$, where $X$ satisfies Hypothesis~\ref{hyp-standing}.
By Theorem~\ref{thm:weakly-closed}, $Z$ corresponds to a principal filter system $\{ \mc{F}_{\alpha} \}$. Thus for each $\alpha$ there is a unique $I_{\alpha} \subseteq O_{\alpha}$ such that $\mc{F}_{\alpha} = \{ J_{\alpha} \, | \, I_{\alpha} \subseteq J_{\alpha} \subseteq O_{\alpha} \}$.  The condition $(\dagger)$ in this case is easily seen to be equivalent to the following:
if $f: O_{\alpha} \to O_{\beta}$ is a morphism, then $I_{\alpha} \subseteq f^{-1}(I_{\beta})$, or equivalently, $f(I_{\alpha}) \subseteq I_{\beta}$.
We give this property a name:
\begin{definition}
Let $X$ have generating set $S = \{O_{\alpha} \, | \, \alpha \in A \}$.  An \emph{ideal} in this set of generators $S$ is a choice of subobject $I_{\alpha} \subseteq O_{\alpha}$ for each $\alpha$, such that for any morphism $f: O_{\alpha} \to O_{\beta}$ in $X$, $f(I_{\alpha}) \subseteq I_{\beta}$.
\end{definition}

For future reference we also state explicitly the conclusion of our discussion above:
\begin{corollary}
\label{cor:idealinprojcase}
Let $X$ satisfy Hypothesis~\ref{hyp-standing}, so has $X$ has generating set $S= \{ O_{\alpha} | \alpha \in A \}$, where each $O_{\alpha}$ is compact and projective.  Then the closed subcategories $Z$ of $X$ are in one-to-one correspondence with ideals in the set of generators $S$, where if $\{ I_{\alpha} \}$ is an ideal then the corresponding closed subcategory $Z$ is generated by $\{O_{\alpha}/I_{\alpha} \}$.
\end{corollary}
\begin{proof}
This follows from Theorem~\ref{thm:weakly-closed}(4) and the definition of ideal.
\end{proof}

\section{(Weakly) closed subcategories and Quotient categories}
\label{sec:quotient}

Let $X$ be a Grothendieck category with localizing subcategory $Y$.  We would like to understand how the (weakly) closed subcategories of $\mc{X} = X/Y$ can be described in terms of those of $X$.  We reiterate that any subcategory defined by specifying its set of objects means the full subcategory with those objects.

First, let us define some special kinds of subcategories of $X$ with good properties in relation to the localizing subcategory $Y$.
\begin{definition}
Let $Z$ be a full subcategory of $X$. We say that $Z$ is \emph{$Y$-essentially stable} if for all $M \in Z$ we also have $\omega \pi(M) \in Z$.  We say that $Z$ is \emph{$Y$-torsionfree generated} if $Z$ is generated by the $Y$-torsionfree objects it contains.  If $Z$ is weakly closed, $Y$-essentially stable, and $Y$-torsionfree generated, we say that $Z$ is \emph{$Y$-weakly closed}.  Similarly, $Z$ is \emph{$Y$-closed} if $Z$ is closed, $Y$-essentially stable, and $Y$-torsionfree generated.
\end{definition}

The term ``$Y$-torsionfree generated" is self-explanatory.  Recall that $\omega \pi(M)$ is the largest essential extension of $M/\tau(M)$ by an object in $Y$.  So if $Z$ is closed under quotients (in particular if $Z$ is weakly closed), then $Z$ is $Y$-essentially stable if and only for any $Y$-torsionfree $M \in Z$, every essential extension of $M$ by an object in $Y$ remains in $Z$.  This explains the terminology ``$Y$-essentially stable".  The terms ``$Y$-weakly closed" and ``$Y$-closed" are explained by the following results.

\begin{proposition}
\label{prop:operations-on-Z}
Let $X$ be a Grothendieck category as above and let $Y$ be a localizing subcategory, with quotient category $\mc{X} = X/Y$.  Let $Z$ be weakly closed in $X$.
\begin{enumerate}
\item  $\pi(Z) = \{ \pi(M) \,| \,M \in Z \}$ is a weakly closed subcategory of $\mc{X}$.
\item If $\mc{Z}$ is weakly closed in $\mc{X}$, then $\pi^{-1}(\mc{Z}) = \{ M \in X \, | \, \pi(M) \in \mc{Z} \}$ is weakly closed in $X$ and $Y$-essentially stable.  
\item  $\wt{Z} = \pi^{-1}(\pi(Z)) = \{ M \in X \,|\, \pi(M) \in \pi(Z) \}$ is weakly closed in $X$ and $Y$-essentially stable.
\item The subcategory $Z'$ generated by all $Y$-torsionfree objects in $Z$ is weakly closed in $X$ and $Y$-torsionfree generated.  $Z'$ is $Y$-essentially stable if and only if $Z$ is.
\item $\wh{Z} = (\wt{Z})'$ is $Y$-weakly closed in $X$.
\end{enumerate} 
\end{proposition}
\begin{proof}
(1) Let $N \in Z$ and write $\pi(N) = \mc{N}$.  If $\mc{M}$ is a subobject of $\mc{N}$, then there is a subobject $M \subseteq N$ such that $\pi(M) = \mc{M}$ \cite[Proposition 4.10]{Kan15}.  Since $M \in Z$, $\mc{M} \in \mc{Z}$.  Then $\mc{N}/\mc{M}$ also represents an arbitrary quotient object of $\mc{N}$.  Now $\mc{N}/\mc{M} = \pi(N)/\pi(M) = \pi(N/M)$ as $\pi$ is exact, and $N/M \in Z$ implies $\mc{N}/\mc{M} \in \mc{Z}$.  We see that $\mc{Z}$ is closed under subquotients.  Given a collection $\{ \mc{M}_{\alpha} \}$ of objects in $\mc{Z}$,  where $\mc{M}_{\alpha} = \pi(M_{\alpha})$ for $M_{\alpha} \in Z$, then 
$\bigoplus \mc{M}_{\alpha} = \bigoplus \pi(M_{\alpha}) \cong \pi(\bigoplus M_{\alpha})$, so $\mc{Z}$ is closed under direct sums because $Z$ is.  

(2) That $\pi^{-1}(\mc{Z})$ is weakly closed is immediate from the facts that $\pi$ is exact and preserves direct sums.
If $M \in \pi^{-1}(\mc{Z})$ then $\pi(\omega\pi(M)) = \pi(M) \in \mc{Z}$
so $\omega \pi(M) \in \pi^{-1}(\mc{Z})$.  Thus $\pi^{-1}(\mc{Z})$ is $Y$-essentially stable.  

(3) This is immediate from (1) and (2).
 
(4) $Z'$ is certainly $Y$-torsionfree generated. Since $Z'$ is generated by a set of objects it is closed under direct sums and quotient objects.  Let $M \in Z'$, so there is an epimorphism $f: \bigoplus K_{\alpha} \to M$ where each $K_{\alpha} \in Z$ is $Y$-torsionfree.  Then $N = \bigoplus K_{\alpha}$ is also $Y$-torsionfree.  Let $P = \ker f$ so that $M \cong N/P$.  If $L$ is a subobject of $M$, then $L \cong Q/P$ for some $P \subseteq Q \subseteq N$, where $Q \in Z$ is again $Y$-torsionfree.  Thus $L$ is in $Z'$ and $Z$ is closed under subobjects.  

Suppose that $Z$ is $Y$-essentially stable.  If $M \in Z' \subseteq Z$, then $\omega \pi(M)  \in Z$.  Note that $\omega \pi(M)$ is $Y$-torsionfree (as is any object in the image of $\omega$), so $\omega \pi(M) \in Z'$ and $Z'$ is $Y$-essentially stable. Conversely, if $Z'$ is $Y$-essentially stable and $M \in Z$ then $M/\tau(M) \in Z'$, so $\omega\pi(M) = \omega\pi(M/\tau(M)) \in Z' \subseteq Z$.

(5) This follows immediately from (3) and (4). 
\end{proof}

In the notation of the preceding proposition, we have the following diagram of weakly closed subcategories associated to $Z$, all of which lie over the same weakly closed subcategory $\mc{Z} = \pi(Z)$ of $\mc{X}$:
\[
\xymatrix{ & \wt{Z} \ar@{-}[dr] &  \\ Z \ar@{-}[ur] & & \wh{Z} = (\wt{Z})' \ar@{-}[dl] \\ & Z' \ar@{-}[ul] & }
\]

\begin{theorem}
\label{thm:quotientweaklyclosed}
Let $X$ be a Grothendieck category with localizing subcategory $Y$, and define $\mc{X} = X/Y$. 
In the notation of Proposition~\ref{prop:operations-on-Z}, there are inverse bijections
\begin{align*}
\{ Y \text{-weakly closed subcategories }\ Z \subseteq X \}  & \longleftrightarrow  \{\text{Weakly closed subcategories}\ \mc{Z} \subseteq \mc{X} \}  \\
Z & \overset{\phi}{\longrightarrow} \pi(Z) \\
(\pi^{-1}(\mc{Z}))' &\overset{\psi}{\longleftarrow} \mc{Z}.
\end{align*}
\end{theorem}
\begin{proof}
It is easy to see that the maps $\phi$ and $\psi$ are well-defined, using Proposition~\ref{prop:operations-on-Z}.

Let $Z \subseteq X$ be $Y$-weakly closed, and let $\wh{Z} = \psi(\phi(Z))$ be generated by all $Y$-torsionfree objects $N$ in $\pi^{-1}(\pi(Z))$.  We claim that $Z = \wh{Z}$.  Since both $Z$ and $\wh{Z}$ are generated by the $Y$-torsionfree objects they contain, it is enough to check that $Z$ and $\wh{Z}$ have the same $Y$-torsionfree objects.  If $N \in Z$ is $Y$-torsionfree, then $N \in \wh{Z}$ by definition.  Conversely, if $P \in \wh{Z}$ is $Y$-torsionfree, then $\pi(P) = \pi(N)$ for some $N \in Z$.  Then $\omega \pi(P) = \omega \pi(N) \in Z$ because $N \in Z$ and $Z$ is $Y$-essentially stable.  Finally, $P$ may be identified with a subobject of $\omega \pi(P)$ since $P$ is $Y$-torsionfree, and thus $P \in Z$.  So $Z = \wh{Z}$ as claimed.

Next, let $\mc{Z} \subseteq \mc{X}$ be weakly closed, and let $\mc{W} = \phi(\psi(\mc{Z})) = \pi(\pi^{-1}(\mc{Z})')$.
 Obviously $\mc{W} \subseteq \mc{Z}$.  On the other hand, if $\mc{M} \in \mc{Z}$ then $\mc{M} = \pi(M)$ where 
 $M = \omega(\mc{M})$ is $Y$-torsionfree, and so $M \in \pi^{-1}(\mc{Z})'$; thus $\mc{M} = \pi(M) \in \mc{W}$.  We conclude that $\mc{Z} = \mc{W}$.  
\end{proof}

Next, we specialize the results above to closed subcategories.  Note that we require $X$ to satisfy (AB4*), that is, that products are exact in $X$, as well as an additional hypothesis on the closed subcategory $Z$ in part (1) of the following result, compared to the weakly closed case.  
\begin{theorem}
\label{thm:quotientclosed}
Let $X$ be a (AB4*) Grothendieck category and let $Y$ be a localizing subcategory, with quotient category $\mc{X} = X/Y$. 
\begin{enumerate}
\item If $Z$ is a $Y$-essentially stable closed subcategory of $X$, then $\pi(Z)$ is closed 
in $\mc{X}$.  
\item If $\mc{Z}$ is closed in $\mc{X}$ then $\wh{Z} = (\pi^{-1}(\mc{Z}))'$ is $Y$-closed in $X$.
\item The bijection of Theorem~\ref{thm:quotientweaklyclosed}(1) restricts to a bijection 
\[
\{ Y \text{-closed subcategories }\ Z \subseteq X \}   \longleftrightarrow  \{\text{closed subcategories}\ \mc{Z} \subseteq \mc{X} \}.  
\]
\end{enumerate}
\end{theorem}
\begin{proof}
(1)  By the correspondence of Theorem~\ref{thm:quotientweaklyclosed}, we know that $\mc{Z} = \pi(Z)$ is at least weakly closed in $\mc{X}$.  
Let $\{ M_{\alpha} \}$ be an arbitrary small collection of objects in $Z$, so $\{ \pi(M_{\alpha}) \}$ is an arbitrary small collection of objects in $\mc{Z}$.  Then $\prod_{\alpha} \pi(M_{\alpha}) = \pi( \prod_{\alpha} \omega\pi(M_{\alpha}))$.   By the $Y$-essentially stable assumption, $\omega \pi(M_{\alpha}) \in Z$ for all $\alpha$.  Since $Z$ is also closed under products, we see that $\prod \pi(M_{\alpha}) \in \mc{Z}$.  Thus $\mc{Z}$ is closed under products, so is a closed subcategory of $\mc{X}$.

(2) Again the correspondence of Theorem~\ref{thm:quotientweaklyclosed} implies that $\wh{Z}$ is at least $Y$-weakly closed in $X$.  Now suppose that $\{ M_{\alpha} \}$ is a collection of objects in $\wh{Z}$.  Since $\wh{Z}$ is generated by $Y$-torsionfree objects in $\pi^{-1}(\mc{Z})$ and a direct sum of $Y$-torsionfree objects in $Y$-torsionfree, $M_{\alpha} \cong N_{\alpha}/P_{\alpha}$ where $N_{\alpha}$ is $Y$-torsionfree with $\pi(N_{\alpha}) \in \mc{Z}$.  By exactness of products in $X$ (AB4*), $\prod_{\alpha} M_{\alpha} \cong (\prod_{\alpha} N_{\alpha})/(\prod_{\alpha} P_{\alpha})$.  Since $N_{\alpha}$ is $Y$-torsionfree, we identify $N_{\alpha}$ with a subobject of $\omega \pi(N_{\alpha})$.     Similarly, exactness of products shows that $\prod_{\alpha} N_{\alpha}$ can be viewed as a subobject of $\prod \omega \pi(N_{\alpha})$. Now $\pi(\prod \omega\pi(N_{\alpha})) = \prod \pi(N_{\alpha}) \in \mc{Z}$ because $\mc{Z}$ is closed under products.  Also, $\prod \omega\pi(N_{\alpha})$ is $Y$-torsionfree since each $\omega \pi(N_{\alpha})$ is.   We see that $\prod \omega\pi(N_{\alpha}) \in \wh{Z}$, and hence its subobject $\prod N_{\alpha}$ is in $\wh{Z}$ as well.  Finally, this implies the epimorphic image $\prod_{\alpha} M_{\alpha}$ is in $\wh{Z}$. So $\wh{Z}$ is closed under products and hence is $Y$-closed.

(3)  This follows immediately from (1) and (2).
\end{proof}

The $Y$-essentially stable property is rather subtle, and it is helpful to have several ways of describing it.  The following notion is relevant for this.
\begin{definition}\label{def:gabrielproduct}
Let $Z$ and $W$ be full subcategories of $X$.  The \emph{Gabriel product} $Z * W$ is the full subcategory of $X$ consisting of all objects $P$ such that there exists a short exact sequence 
$0 \to M \to P \to N \to 0$ where $M \in W$ and $N \in Z$.  
\end{definition}
Note that the order of the categories in the notation is chosen so that an object in $Z * W$ corresponds to 
an element of $\Ext^1_X(N,M)$ for some $N \in Z$, $M \in W$.  If $Z$ and $W$ are both weakly closed or both closed, then $Z * W$ has the same property; see \cite[Proposition 3.3.6]{VdB01}. 

\begin{proposition}
\label{prop:Y-essentiallystable}
Let $X$ be a Grothendieck category with localizing subcategory $Y$.  Let $Z$ be a weakly closed subcategory of $X$.  As above, let $\pi: X \to X/Y$ be the quotient map, let $Z'$ be the subcategory of $Z$ generated by its $Y$-torsionfree objects, and let $\wh{Z} = (\pi^{-1}(\pi(Z)))'$.
\begin{enumerate}
\item $\wh{Z} = (Y * Z)'$.
\item The following are equivalent:
\begin{enumerate}
    \item[(i)] $Z$ is $Y$-essentially stable.
    \item[(ii)] $Z' = \wh{Z}$.      
    \item[(iii)] $Y * Z \subseteq Z * Y$.
\end{enumerate}
\end{enumerate}
\end{proposition}
\begin{proof}
(1) It is enough to show that $Y * Z$ and $\pi^{-1}(\pi(Z))$ have the same set of $Y$-torsionfree objects.
Suppose that $M$ is a $Y$-torsionfree object in $Y * Z$, so there is $L \subseteq M$ with $L \in Z$ and $M/L \in Y$.  Then $\pi(M) = \pi(L) \in \pi(Z)$, so $M \in \pi^{-1}(\pi(Z))$.  Conversely, suppose that $M$ is a $Y$-torsionfree object such that $\pi(M) \in \pi(Z)$.  There is some $N \in Z$ with $\pi(N) = \pi(M)$.  Noting that $\pi(N/\tau(N)) = \pi(N)$, we can assume that $N$ is $Y$-torsionfree.  Then $\omega \pi(M) = \omega \pi(N)$ and both $M$ and $N$ can be identified with subobjects of $\omega \pi(N)$.  As such, we have $M \cap N \subseteq M$ with $M \cap N \in Z$ and $M/(M \cap N) \cong (M+N)/N \subseteq (\omega \pi(N))/N \in Y$, so $M \in Y * Z$.

(2)
$(i) \implies (ii)$:  By (1), we need to show $Z' = (Y * Z)'$. 
If $M$ is $Y$-torsionfree and in $Y * Z$, choose $0 \subseteq L \subseteq M$ with $L \in Z$ and $M/L \in Y$.  The extension $L \subseteq M$ is essential, for if $N \subseteq M$ with $N \cap L = 0$, then $N$ embeds in $M/L \in Y$, but $N$ is $Y$-torsionfree, so $N = 0$.  Since $\omega\pi(L)$ is the unique largest essential extension of $L$ by an element of $Y$, $M$ embeds in $\omega \pi(L)$.  By hypothesis, since $L \in Z$, we have $\omega \pi(L) \in Z$, so $M \in Z$.  This shows that every 
generator of $(Y * Z)'$ is a generator of $Z'$, so $(Y * Z)' \subseteq Z'$.  The other inclusion is immediate.

$(ii) \implies (iii)$: Using part (1) again, the assumption is $(Y * Z)' = Z'$.
Let $M \in Y * Z$.  Then $M/\tau(M) \in (Y * Z)' = Z' \subseteq Z$, so the inclusion $0 \subseteq \tau(M) \subseteq M$ shows that $M \in Z * Y$.

$(iii) \implies (i)$: Let $N \in Z$.  Now $\omega \pi(N)$ is the largest essential extension of $N' = N/\tau(N) \in Z$ by an object in $Y$.  So $\omega \pi(N) \in Y * Z$.  Then $\omega \pi(N) \in Z * Y$, and since $\omega \pi(N)$ is $Y$-torsionfree, $\omega \pi(N) \in Z$.  This shows that $Z$ is $Y$-essentially stable by definition.
\end{proof}

We have now seen that to understand (weakly) closed subcategories of a quotient category $X/Y$, we can instead study $Y$-(weakly) closed subcategories of $X$.  In Section~\ref{sec:generators} we saw how to describe (weakly) closed subcategories of $X$ in terms of filter systems in a set of generators.  In the analysis of examples, it can be useful to interpret the properties of being $Y$-essentially stable and $Y$-torsionfree generated in terms of these filter systems, which we do in the next results. 

\begin{definition} 
\label{def:saturation}
Given a subobject $J$ of a fixed generator $O_{\alpha}$ of $X$ we define the \emph{$Y$-saturation} of $J$ (inside $O_{\alpha}$) to be the subobject $\wt{J}$ of $O_{\alpha}$ such that $\wt{J}/J = \tau(O_{\alpha}/J)$.  We say that $J$ is \emph{saturated} (inside $O_{\alpha}$) if $J = \wt{J}$.
\end{definition}
It is elementary to see that if $I \subseteq J \subseteq O_{\alpha}$, then $\wt{I} \subseteq \wt{J}$.  

\begin{proposition}
\label{prop:Y-reduced}
Let $X$ be a Grothendieck category with a small set of compact projective generators $S = \{ O_{\alpha} \}$.  Let $Y$ be a localizing subcategory of $X$.  Let $Z$ be a weakly closed subcategory of $X$ corresponding to the filter system $\{ \mc{F}_{\alpha} \}$ in the set of generators $S$.  Let $Z'$ be generated by the $Y$-torsionfree objects in $Z$, and let $\wh{Z} = (\pi^{-1}(\pi(Z)))'$.
 \begin{enumerate}
    \item The filter system corresponding to $Z'$ is $\{ \mc{F}'_{\alpha} \}$, 
where 
\[
\mc{F}'_{\alpha} = \{ J \subseteq O_{\alpha} | \wt{I} \subseteq J\ \text{for some}\ I \in \mc{F}_{\alpha} \}.
\]

\item The filter system corresponding to $\wh{Z}$ is $\{ \wh{\mc{F}}_{\alpha} \}$, where 
\[
\wh{\mc{F}}_{\alpha} = \{ J \subseteq O_{\alpha} \, | \, \text{there is}\ I \subseteq L \subseteq O_{\alpha}\ \text{with}\ O_{\alpha}/L \in Y, L/I \in Z, \text{and}\ \wt{I} \subseteq J\}.
\]
\end{enumerate}
\end{proposition}
\begin{proof}
(1) If $J \in \mc{F}'_{\alpha}$, say $\wt{I} \subseteq J \subseteq O_{\alpha}$ with $I \in \mc{F}_{\alpha}$, then 
$O_{\alpha}/\wt{I}$ is in $Z$ and is $Y$-torsionfree.  Since $O_{\alpha}/J$ is an epimorphic image of $O_{\alpha}/\wt{I}$, 
$O_{\alpha}/J \in Z'$.

Conversely, suppose that $J \subseteq O_{\alpha}$ with $O_{\alpha}/J \in Z'$.  There is a set of $Y$-torsionfree objects in $Z$, say $\{ M_{\alpha} \}$, and an epimorphism $g: M = \bigoplus_{\alpha} M_{\alpha} \to O_{\alpha}/J$.  The object $M$ is also $Y$-torsionfree and in $Z$.  Let $p: O_{\alpha} \to O_{\alpha}/J$ be the canonical epimorphism.  Using the projectivity of $O_{\alpha}$, we may find a morphism $f: O_{\alpha} \to M$ such that $g \circ f = p$.  Let $I = \ker f$, and note that $I \subseteq J = \ker p $.  Since $M$ is $Y$-torsionfree and in $Z$, $O_{\alpha}/I$ is $Y$-torsionfree and in $Z$. That is, $I = \wt{I}$ and $I \in \mc{F}_{\alpha}$.  So $\wt{I} \subseteq J$ where $I \in \mc{F}_{\alpha}$.

(2) We have seen that $\wh{Z} = (Y * Z)'$ in Proposition~\ref{prop:Y-essentiallystable}.  We know that $Y * Z$ is weakly closed in $X$ since both $Y$ and $Z$ are.  Let $\{ \mc{G}_{\alpha} \}$ be the filter system corresponding to $Y * Z$.  It is obvious that 
\[
\mc{G}_{\alpha} = \{ J \subseteq O_{\alpha} \, | \, \text{there is}\ J \subseteq L \subseteq O_{\alpha}\ \text{with}\ O_{\alpha}/L \in Y, L/J \in Z \}.
\]
Now note that $\wh{F}_{\alpha} = \mc{G}'_{\alpha}$ and apply (1).
\end{proof}

\begin{corollary}
\label{cor:Y-reduced}
Assume the same hypotheses as in Proposition~\ref{prop:Y-reduced}, so $\{ \mc{F}_{\alpha} \}$ is the filter system corresponding to a weakly closed subcategory $Z$.
\begin{enumerate}
    \item $Z$ is $Y$-torsionfree generated if and only if for all $J \in \mc{F}_{\alpha}$, there is $I \in \mc{F}_{\alpha}$ with $\wt{I} \subseteq J$.
    \item $Z$ is $Y$-essentially stable if and only if whenever $I \subseteq L \subseteq O_{\alpha}$ with $L/I \in Z$ and $O_{\alpha}/L \in Y$, then $O_{\alpha}/\wt{I} \in Z$.
\end{enumerate}
\end{corollary}
\begin{proof}
(1) $Z$ is $Y$-torsionfree generated if and only if $Z = Z'$, so this follows from Proposition~\ref{prop:Y-reduced}(1).

(2) $Z$ is $Y$-essentially stable if and only if $Z' = \wh{Z}$ by Proposition~\ref{prop:Y-essentiallystable}, so this is a consequence of Proposition~\ref{prop:Y-reduced}(2).
\end{proof}

Kanda has given a bijection between weakly closed (or in his terminology, prelocalizing) subcategories of a quotient category $X/Y$ and a different class of weakly closed subcategories of $X$ \cite[Proposition 4.14(1)]{Kan15}: in the notation above, $\mc{Z} \subseteq X/Y$ simply corresponds to $\wt{Z} = \pi^{-1}(\mc{Z}) \subseteq X$. This is a very natural correspondence, but it does not restrict well to closed subcategories.  In fact, in many examples of interest, $X$ is an inverse limit of objects in $Y$, and when this is the case, any closed subcategory containing $Y$ is all of $X$.  On the other hand, Kanda's bijection restricts to localizing subcategories, and implies that the localizing subcategories of $X/Y$ are in one-to-one correspondence with localizing subcategories of $X$ which contain $Y$ \cite[Proposition 4.14(3)]{Kan15}. 

In the last result of this section we record for reference how some of the theory above plays out for a pair of localizing subcategories, which will be useful in one of the examples below.  Let $X$ be a Grothendieck category with two localizing subcategories $Y_1 \subseteq Y_2$.  Let $\pi_1: X \to X/Y_1$ and $\omega_1: X/Y_1 \to X$ be the quotient and section functors.  Then $\mc{Y} = \pi_1(Y_2)$ is localizing in $X/Y_1$ \cite[Proposition 4.14(3)]{Kan15}. Similarly, let $\wt{\pi}: X/Y_1 \to (X/Y_1)/\mc{Y}$ and $\wt{\omega}:(X/Y_1)/\mc{Y} \to X/Y_1$  be the quotient and section functors for this quotient category. The composite $\wt{\pi} \circ \pi_1: X \to (X/Y_1)/\mc{Y}$ is exact with fully faithful right adjoint $\omega_1 \circ \wt{\omega}$, which implies by the Gabriel-Popescu theorem (see the next section) that $(X/Y_1)/\mc{Y}$ is equivalent to $X/Y_2$, since $Y_2 = \{ M \in X | \wt{\pi} \circ \pi_1(M) = 0 \}$.  Using this we write $\wt{\pi} \circ \pi_1 = \pi_2$ and $\omega_1 \circ \wt{\omega} = \omega_2$ and identify these with the quotient and section functors for the quotient of $X$ by $Y_2$.  See \cite[Proposition 4.18(3)]{Kan15} for more details.  

\begin{proposition}
\label{prop:twolocalizing}
Keep all of the notation from the previous paragraph.  Let $Z$ be weakly closed in $X$.
\begin{enumerate}
\item  If $M \in X$ is $Y_2$-torsionfree, then $\pi_1(M) \in X/Y_1$ is $\mc{Y}$-torsionfree.  
\item If $Z$ is $Y_2$-torsionfree generated in $X$ then $\pi_1(Z)$ is $\mc{Y}$-torsionfree generated in $X/Y_1$.
\item  If $Z$ is $Y_2$-essentially stable in $X$, then $\pi_1(Z)$ is $\mc{Y}$-essentially stable in $X/Y_1$.
\end{enumerate}
\end{proposition}
\begin{proof}
(1).  Let $\mc{N} \subseteq \pi_1(M)$ with $\mc{N} \in \mc{Y}$.  Then we can find $N \subseteq M$ 
such that $\pi_1(N) = \mc{N}$ \cite[Proposition 4.10]{Kan15}.  Now $N \in \pi_1^{-1}(\mc{Y}) = Y_2$, by \cite[Proposition 4.14(3)]{Kan15}.  Since $M$ is $Y_2$-torsionfree, $N = 0$ and so $\mc{N} = \pi_1(N) = 0$.

(2).  Given a set $\{M_{\alpha} \}$ of $Y_2$-torsionfree generators of $X$, then $\{ \pi_1(M_{\alpha}) \}$ is a set of $\mc{Y}$-torsionfree generators of $X/Y_1$ by (1).

(3). We need $\wt{\omega}\wt{\pi}(\pi_1(Z)) \subseteq \pi_1(Z)$.  We calculate 
\[
\wt{\omega} \wt{\pi} \pi_1 = \wt{\omega} \pi_2 =  \pi_1 \omega_1 \wt{\omega} \pi_2 = \pi_1 \omega_2 \pi_2.
\]
Since $Z$ is $Y_2$-essentially stable in $X$, we have $\omega_2 \pi_2(Z) \subseteq Z$.  
Thus $\wt{\omega} \wt{\pi} \pi_1(Z) = \pi_1 \omega_2 \pi_2(Z) \subseteq \pi_1(Z)$.
\end{proof}

\section{Examples}
\label{sec:examples}
We now give a series of examples to illustrate the subtle behavior of closed subcategories under taking quotient categories.    

By the Gabriel-Popescu theorem \cite[Theorem 3.7.9, Corollary 4.10]{Pop73}, any Grothendieck category $\mc{X}$, say with generator $O$, is equivalent to a quotient category $(\rcatMod R)/Y$ for the ring $R = \End_X(O)$ and some localizing subcategory $Y$.  Since $\rcatMod R$ has exact products, this shows that in theory, Theorem~\ref{thm:quotientclosed} can be applied to understand the closed subcategories of any Grothendieck category.  In practice, it may be hard to understand the nature of $Y$ and to determine which closed subcategories $\rcatMod (R/I)$ for ideals $I$ of $R$ are $Y$-closed.  Also, since $O$ is often a very large object, the ring $R$ is also then very large and awkward to work with, which is why it was useful to formulate the results in Section~\ref{sec:generators} in terms of a set of compact projective generators instead of a single generator.  Still, this suggests that to look for examples of various kinds of behavior in the results of Section~\ref{sec:quotient}, we can focus on quotient categories of module categories.

In the following examples we fix $X = \rcatMod R$ for a unital ring $R$.  We fix the following general notation in all of the examples as well: Let $Y$ be a localizing subcategory of $X$, corresponding to a Gabriel filter $\mc{F}$ of right ideals of $R$.  Let $Z$ be a closed subcategory of $X$, so $Z = \rcatMod (R/K)$ for an ideal $K$ of $R$.  For any right ideal $J$, let $\wt{J}$ be its saturation in $R$  with respect to $Y$, so $\wt{J}/J = \tau(R/J)$ is the largest subobject of $R/J$ in $Y$.  
By Corollary~\ref{cor:Y-reduced}, $Z$ is $Y$-torsionfree generated if and only if $K = \wt{K}$, while
$Z$ is $Y$-essentially stable if and only if for any $I \in \mc{F}$ and any right ideal $L \subseteq I$ such that $I/L \in Z$, we have $R/\wt{L} \in Z$.  Because belonging to $Z$ is the same as being annihilated by $K$, for fixed $I$ the set of possible $L$ has a unique smallest member $L = IK$.  So $Z$ is $Y$-essentially stable if and only if $R/\wt{IK} \in Z$ for all $I \in \mc{F}$.  This is equivalent to $K \subseteq \wt{IK}$ for all $I \in \mc{F}$.  Finally, $K \subseteq \wt{IK}$ is equivalent to $\wt{IK} = \wt{K}$.

To summarize, $Z = \rcatMod (R/K)$ is $Y$-torsionfree generated if and only if $K = \wt{K}$; $Y$-essentially stable if and only if $\wt{K} = \wt{IK}$ for all $I \in \mc{F}$; and thus $Y$-closed if and only if $K = \wt{K} = \wt{IK}$ or equivalently $K = \wt{IK}$ for all $I \in \mc{F}$.
Therefore, by Theorem~\ref{thm:quotientclosed}, the closed subcategories of $X/Y$ correspond to ideals $K$ of $R$ such that $\wt{IK} = K$ for all $I \in \mc{F}$.  

\begin{example}
\label{ex:commutative}
Suppose that $R$ is commutative.  In this case, for any $I \in \mc{F}$, clearly $K/IK = K/KI \in Y$.  This shows that $K \subseteq \wt{IK}$.  Thus $Z$ is automatically $Y$-essentially stable, and the closed subcategories of $X/Y$ simply correspond to the ideals $K$ of $R$ such that $K = \wt{K}$.  This shows that the $Y$-essentially stable condition is a purely noncommutative phenomenon in some sense.
\end{example}

On the other hand, it is not hard to find a simple example of a noncommutative ring $R$ and a closed subcategory of $\rcatMod R$ that is not $Y$-essentially stable for some localizing subcategory $Y$.

\begin{example}
\label{ex:noncommutative} [cf. \cite[Example 6.11]{Smi02}]
Let $R = \begin{pmatrix} \kk & \kk \\ 0 & \kk \end{pmatrix}$ be the ring of upper triangular matrices over a field $\kk$. Consider the ideals $I_1 = \begin{pmatrix} 0 & \kk \\ 0 & \kk \end{pmatrix}$ and $I_2 =\begin{pmatrix} \kk & \kk \\ 0 & 0 \end{pmatrix}$ of $R$, so  
$S_1 = R/I_1$ and $S_2 = R/I_2$ are the two simple right $R$-modules.  As right modules, 
$R = P_1 \oplus P_2$ for the two indecomposable projectives $P_1$ and $P_2$ given by the rows of $R$.  There is a non-split exact sequence $0 \to S_2 \to P_1 \to S_1 \to 0$, while $P_2 \cong S_2$.

Since $I_i$ is an ideal, $Z_i = \rcatMod (R/I_i)$ is closed in $X$ for $i = 1, 2$. 
The category $Z_i$ is semisimple, consisting of all direct sums of copies of $S_i$.  The simple module $S_i$ has no self-extensions, so $Z_i$ is also localizing for $i = 1, 2$.

Now taking $Y = Z_1$ and $Z = Z_2$, then $Z$ is closed but not $Y$-essentially stable, because 
$Z = Z'$ while $P_1 \in \wh{Z}$ so $\wh{Z} = X$.    Similarly, one sees that the semisimple closed subcategory $\rcatMod (R/(I_1 \cap I_2))$ consisting of all direct sums of $S_1$ and $S_2$ (which is $Z_1 \cup Z_2$ in the language of the next section) is not $Y$-essentially stable.
\end{example}

\begin{example}
\label{ex:ringloc}
Let $R$ be any ring.  Suppose that $S \subseteq R$ is a \emph{right denominator set} 
as in \cite[Chapter 10]{GW04}; so the Ore localization $RS^{-1}$ is defined.  
In the language of quotient categories, the full subcategory 
\[
Y = \{ M \in \rcatMod R \, | \,  \text{for all}\ m \in M\ \text{there is}\ s \in S\ \text{such that}\ ms = 0\}
\] 
is a localizing subcategory; the corresponding filter of right ideals of $R$ is $\mc{F} = \{ I \subseteq R \,|\, s \subseteq I\ \text{for some}\ s \in S \}$.    It is well known that $X/Y$ can be identified with $\rcatMod RS^{-1}$ in such a way that $\pi: X \to X/Y$ is identified with the functor $- \otimes_R RS^{-1}: \rcatMod R \to \rcatMod RS^{-1}$.

We claim that the following are equivalent for an ideal $K$: (i) $Z = \rcatMod (R/K)$ is $Y$-essentially stable; 
(ii) given $k \in K$ and $s \in S$, there exists $t \in S$ and $k' \in K$ with $sk' = k t$; and (iii) $KS^{-1}$ is an ideal of $RS^{-1}$.  
We already saw that $Z$ is $Y$-essentially stable if and only if $K \subseteq \wt{IK}$ for any $I \in \mc{F}$.  Given $s \in S$ and $k \in K$, taking $I = sR$ we have $IK = sK$; then the right annihilator of $k + sK$ in $K/sK$ should be in $\mc{F}$, so there is $t \in S$ and $k' \in K$ with $kt = sk'$.  This shows that (i) $\implies$ (ii) and the converse is similar.  If we have (ii), then given $k \in K$ and $s \in S$, in the ring $RS^{-1}$ we have $s^{-1}k = k' t^{-1}$ for some $k' \in K$ and $t \in S$; so $S^{-1}K \subseteq KS^{-1}$ and thus $RS^{-1}KS^{-1} \subseteq RK S^{-1} \subseteq KS^{-1}$.  So (ii) $\implies$ (iii)  and again we omit the converse.

If $R$ is a right noetherian ring, then for any ideal $K$ of $R$ it is known that $KS^{-1}$ is again an ideal of $RS^{-1}$ \cite[Theorem 10.18]{GW04}.  In this case every closed subcategory $Z = \rcatMod (R/K)$ of $\rcatMod R$ is $Y$-essentially stable, and Theorem~\ref{thm:quotientclosed} just recovers the well known bijection between ideals of $RS^{-1}$ and ideals $K$ of $R$ with $K = \wt{K}$.  
\end{example}

\begin{example}
\label{ex:idealpowers}
Let $X = \rcatMod R$, let $I$ be any ideal in $R$, and let $Y$ be the weakly closed subcategory associated to the filter of right ideals $\mc{F} = \{ J \subseteq R \, | \, I^n \subseteq J\ \text{for some}\ n \geq 1 \}$.  Assume that $Y$ is a localizing subcategory, or equivalently that $\mc{F}$ is a Gabriel filter.  This is automatic if $R$ is noetherian, as is easy to check.

Let $Z = \rcatMod (R/K)$ for some ideal $K$.  We know by Proposition~\ref{prop:Y-reduced} that the associated $Y$-weakly closed subcategory $\wh{Z}$ corresponds to the filter of right ideals
\[
\wh{F} = \{ J \subseteq R \,| \, \wt{P} \subseteq J\ \text{for some}\ P \subseteq L \subseteq R\ \text{with}\ L/P \in Z, R/L \in Y \}.
\]
So if $J \in \wh{F}$ then $I^n \subseteq L$ for some $n$, while $LK \subseteq P$; thus $I^n K \subseteq P$ and so $\wt{I^n K} \subseteq J$.  It follows that $\wh{F} = \{ J \, | \, \wt{I^nK} \subseteq J\ \text{for some}\ n \geq 0 \}$.  
Because the filter is determined by all objects containing an object from the descending chain 
\begin{equation}
\label{eq:chain}
\wt{K} \supseteq \wt{IK} \supseteq \wt{I^2K} \supseteq \dots,
\end{equation}
the filter is principal if and only if \eqref{eq:chain} stabilizes.

Thus there are two possibilities.  If the descending chain \eqref{eq:chain} stabilizes, then $\wh{F}$ is principal with unique minimal element $\wt{I^nK}$ for some $n$, so $\wh{Z} = \rcatMod (R/\wt{I^nK})$ is $Y$-closed in $X$, and $\pi(Z) = \pi(\wh{Z})$ is closed in $X/Y$ by Theorem~\ref{thm:quotientclosed}.  On the other hand, we have seen that $Z$ is $Y$-essentially stable if and only if $\wt{K} = \wt{I^nK}$ for all $n$; that is, if all terms in the chain \eqref{eq:chain} are equal.  This is a more stringent condition which guarantees that $\wh{Z} = Z' = \rcatMod (R/\wt{K})$.  

If \eqref{eq:chain} does not stabilize, then $\wh{Z}$ is not closed in $X$, so $\pi(Z)$ is not closed in $X/Y$, only weakly closed, by the correspondence of Theorem~\ref{thm:quotientclosed}.  

To see what happens in an explicit case, consider Example~\ref{ex:noncommutative} again, and let $I = I_1$ and $K = I_2$ in the notation of that example. Then $Y = Z_1 = \rcatMod (R/I_1)$ is localizing of the form above (since $I_1^2 = I_1$). Although $Z = \rcatMod(R/I_2)$ is not $Y$-essentially stable, $\pi(Z) =  X/Y$ is nonetheless closed in $X/Y$.  We have $\wt{K} = K$ while $\wt{I^n K} = 0$ for all $n \geq 1$, so the chain \eqref{eq:chain} does stabilize, but not until the second term.
\end{example}

\begin{example}
\label{ex:badloc}
Consider the special case of Example~\ref{ex:idealpowers} where $I = sR$ is principal, generated by a normal nonzerodivisor $s$.  It is easy to check that $Y$ is a localizing subcategory in this case whether or not $R$ is noetherian.  In fact $S = \{1, s, s^2, \dots \}$ is a right denominator set in $R$ and so this is also a special case of Example~\ref{ex:ringloc}.  

Let $\phi: R \to R$ be the automorphism such that $sr = \phi(r) s$ for all $r \in R$.
For an ideal $K$ of $R$, note that $\wt{s^nK} = \wt{\phi^n(K) s^n} = \wt{\phi^n(K)}$.  Moreover, since $\phi(S) = S$, the automorphism $\phi$ commutes with saturation and so $\wt{\phi^n(K)} = \phi^n(\wt{K})$.  Thus the descending chain \eqref{eq:chain} takes the form
\[
\wt{K} \supseteq \phi(\wt{K}) \supseteq \phi^2(\wt{K}) \supseteq \dots
\]
Since $\phi$ is an isomorphism, if the chain stabilizes then all elements of the chain are equal. In this case $Z = \rcatmod (R/K)$ is $Y$-essentially stable, $\wh{Z} = \rcatmod (R/\wt{K})$, and $\mc{Z} = \pi(Z)$ is closed in $X/Y$.  Otherwise the chain is properly descending, in which case 
$\wh{Z}$ and $\pi(Z)$ are only weakly closed, not closed.   In the latter case the chain $\wt{K} \subseteq \phi^{-1}(\wt{K}) \subseteq \phi^{-2}(\wt{K}) \subseteq \dots$ must be proper ascending as well, reconfirming that this can only happen for non-noetherian rings $R$ (see Example~\ref{ex:ringloc}).

The standard example  of a ring $R$ with an ideal $K$ and a localization $RS^{-1}$ such that $KS^{-1}$ is only a right ideal is exactly of this latter type \cite[Exercise 10L]{GW04}.  Explicitly, let $\kk$ be a field and let $R = \kk[x_n | n \in \mb{Z} ][s; \sigma]$ 
where $\sigma(x_n) = x_{n+1}$ for all $n$, and let $K = \wt{K} = \sum_{n \geq 0} x_n R$.  Obviously $K \supsetneq sKs^{-1} = \sum_{n \geq 1} x_n R$.
\end{example}

Example~\ref{ex:badloc} shows that the $Y$-essentially stable hypothesis in Theorem~\ref{thm:quotientclosed}(1) cannot be removed. By considering more general kinds of quotient categories not related to Ore localization, we can easily give such an example where $X = \rcatMod R$ for a noetherian ring $R$.
\begin{example}
\label{ex:pinotclosed}
Consider Example~\ref{ex:idealpowers} for a noetherian ring $R$, where $I$ is an arbitrary ideal, but assume that $K = xR$ for a normal nonzerodivisor $x \in R$. Let $\psi: R \to R$ be the automorphism such that $xz = \psi(z) x$ for $z \in R$.  In this case 
$I^n K = I^n x = x \psi^{-1}(I^n)$.  So the descending chain in \eqref{eq:chain} becomes 
\begin{equation}
\label{eq:chain2}
\wt{xR} \supseteq \wt{x\psi^{-1}(I)} \supseteq \wt{x \psi^{-1}(I^2)} \supseteq \dots.
\end{equation}
Suppose that this chain stabilizes, say $\wt{x \psi^{-1}(I^n)} = \wt{x \psi^{-1}(I^{n+1})}$ for some $n$. Since $R$ is noetherian, this means there is $m \geq 1$ such that 
$x \psi^{-1}(I^n) I^m \subseteq x \psi^{-1}(I^{n+1})$, or equivalently $I^n \psi(I^m) \subseteq I^{n+1}$.

Now suppose that $\psi(I)$ and $I$ are comaximal, i.e. $\psi(I) + I = R$, and that $I^n \supsetneq I^{n+1}$ for all $n \geq 1$.  Then $\psi(I)^m = \psi(I^m)$ and $I$ are also comaximal for any $m$. The equation $I^n \psi(I^m) \subseteq I^{n+1}$ implies $I^n (\psi(I^m) + I) \subseteq I^{n+1}$ and so $I^n \subseteq I^{n+1}$, a contradiction.  So \eqref{eq:chain2} properly descends.  Then by Example~\ref{ex:idealpowers}, $Z = \rcatMod (R/K)$ is a closed subcategory of $X = \rcatMod R$ for which $\pi(Z) \subseteq X/Y$ is only weakly closed, not closed.

For an explicit example of the scenario above, let $\kk$ be a field, let $R = \kk \langle x, y \rangle/(yx - qxy)$ be the quantum plane for some $1 \neq q \in \kk$, and let $I = (x, y-1)$ and $K = xR$.
\end{example}

Our last example in this section is perhaps the most subtle.  It shows the importance of the (AB4*) condition in Theorem~\ref{thm:quotientclosed}.

\begin{example}
\label{ex:weird}
Keep the setup in Example~\ref{ex:pinotclosed}, so $K = xR$ for a normal nonzerodivisor $x$ in a noetherian ring $R$, $\psi$ is the automorphism of $R$ with $x r = \psi(r) x$, and $I$ is an ideal of $R$ such that $I^n \neq I^{n+1}$ for all $n \geq 0$ and $I$ and $\psi(I)$ are comaximal. Suppose now in addition that $xR$ is a prime ideal, and that $x^2$ is central in $R$, so $\psi^2 = 1$.  Let $J = I \cap \psi(I)$, so that $J$ is a $\psi$-fixed ideal of $R$.

Let $Y_1$ be the localizing subcategory associated to $I$ as in Example~\ref{ex:idealpowers}.  Similarly, let $Y_2$ be the localizing subcategory associated to $J$.  Let $\pi_i: X \to X/Y_i$ and $\omega_i: X/Y_i \to X$ be the quotient and section functors for each $i$.  Note that $Y_1 \subseteq Y_2$, so we are now in the situation discussed in Proposition~\ref{prop:twolocalizing}.  Thus $\mc{Y} = \pi_1(Y_2)$ is localizing in $X/Y_1$, and putting $\wt{\pi}: X/Y_1 \to (X/Y_1)/\mc{Y}$ and  $\wt{\omega}: (X/Y_1)/\mc{Y} \to X/Y_1$ for the quotient and section functors, we can identify $X/Y_2$ with $(X/Y_1)/\mc{Y}$ so that 
$\pi_2 = \wt{\pi} \circ \pi_1$ and $\omega_2 = \omega_1 \circ \wt{\omega}$.

Let $Z = \rcatMod R/K$, which is closed in $X = \rcatMod R$.  Let $Z_1 = \pi_1(Z) \subseteq X/Y_1$ and $Z_2 = \wt{\pi}(Z_1) = \pi_2(Z) \subseteq X/Y_2$.
It is immediate from Example~\ref{ex:pinotclosed} that $Z_1$ is not closed in $X/Y_1$, only weakly closed.  On the other hand, since $J$ is $\psi$-fixed, $\wt{x \psi^{-1}(J^n)} = \wt{xJ^n} = \wt{xR}$ for all $n$ (where $\wt{L}$ here means saturation of $L$ with respect to $Y_2$), so that \eqref{eq:chain2} stabilizes immediately, or equivalently $Z$ is $Y_2$-essentially stable.  So $Z_2 = \pi_2(Z)$ is closed in $X/Y_2$.

We claim that $Z$ is $Y_2$-torsionfree generated in $X$.  It suffices to show that $R/K$ is $Y_2$-torsionfree.  Suppose we have $z + K \in R/K$ such that $(z +K)J^n = 0$ for some $n \geq 1$.  Then $zJ^n \subseteq K$ and since $K$ is prime, either $z \in K$ or $J^n \subseteq K$.  In the latter case $J \subseteq K$ and thus $I \psi(I) \subseteq K$, so $I \subseteq K$ or $\psi(I) \subseteq K$.  But since $\psi(K) = K$ and $\psi^2 = 1$, this implies that both $I \subseteq K$ and $\psi(I) \subseteq K$, contradicting $I + \psi(I) = R$.
We conclude that $z \in K$ and so $z + K = 0$ as required.

Now applying Proposition~\ref{prop:twolocalizing}, we see that $Z_1 = \pi_1(Z)$ is both $\mc{Y}$-essentially stable and $\mc{Y}$-torsionfree generated.  That is, $Z_1$ is $\mc{Y}$-weakly closed in $X/Y_1$.  Since $\wt{\pi}(Z_1) = Z_2$, $Z_1$ and $Z_2$ are corresponding weakly closed subcategories under the correspondence of Proposition~\ref{thm:quotientweaklyclosed}, applied to $X/Y_1$ and its quotient category $(X/Y_1)/\mc{Y}$.  However, $Z_1$ is not closed, while $Z_2$ is.

Again it is easy to provide an explicit example.   Let $\kk$ be a field, let $R = \kk \langle x, y \rangle/(yx - qxy)$ be the quantum plane for $q = -1$, and let $I = (x, y-1)$.
\end{example}

Example~\ref{ex:weird} shows that the (AB4*) hypothesis is essential for the proof of Theorem~\ref{thm:quotientclosed}(2).  Indeed, the only way to make sense of Example~\ref{ex:weird} is to conclude that $X/Y_1$ must fail to satisfy (AB4*).  This is not in itself surprising, as general quotient categories of (AB4*) categories very often lose this property. 
More important, the example shows that one can not expect a nice correspondence between closed subcategories of a Grothendieck category $\mc{X}$ and its quotient category $\mc{X}/\mc{Y}$ when $\mc{X}$ does not satisfy (AB4*).   Of course one may always express $\mc{X}$ as a quotient category $X/Y_1$ for some (AB4*) category $X$, for example using the Gabriel-Popescu theorem, and then the closed subcategories of $\mc{X}/\mc{Y}$
can be analyzed by expressing $\mc{X}/\mc{Y} \simeq X/Y_2$ for some $Y_2$.

\section{Intersections and unions}
\label{sec:operations}

Intersection and union are basic operations we can perform on (weakly) closed subcategories, which are studied extensively by Smith in \cite{Smi02}.  In this section we work out the interaction of these operations with our results on quotient categories.  In particular, we will give a surprising example showing that the union operation does not necessarily preserve the property of being closed.  At the end of the section, we describe in more detail how our results relate to those of Smith.

We begin with the intersection operation, which is the more straightforward one.  Given a (small) collection of full subcategories $\{ Z_{\beta} \}$ of a Grothendieck category $X$, we define $\bigcap_{\beta} Z_{\beta}$ to be the full subcategory consisting of objects in all of the $Z_{\beta}$.  It is easy to see directly from the definitions that if all of the $Z_{\beta}$ are weakly closed (resp. closed), then $\bigcap_{\beta} Z_{\beta}$ is also weakly closed (resp. closed).  

In the setting of Section~\ref{sec:generators}, we can easily describe the intersection operations of weakly closed (resp. closed) subcategories in terms of their filter systems (resp. ideals).  The proof of the following result is routine.
\begin{proposition} Let $X$ be a Grothendieck category with a small set $S = \{ O_{\alpha} \}_{\alpha \in A}$ of compact projective generators.  Let $\{ Z_{\beta} \}$ be a small set of weakly closed subcategories of $X$, where $Z_{\beta}$ corresponds to the filter system $\{ \mc{F}^{\beta}_{\alpha} \}$ in the set of generators $S$.  Let $Z = \bigcap_{\beta} Z_{\beta}$.
\label{prop:int-gen}
\begin{enumerate}
    \item $Z$ has corresponding filter system $\{ \mc{F}_{\alpha} \}$ where $\mc{F}_{\alpha} = \{ I \subseteq O_{\alpha} | 
    I \in \mc{F}^{\beta}_{\alpha}\ \text{for all}\ \beta \}$.
    \item If every $Z_{\beta}$ is closed in $X$, say with corresponding ideal $\{ I^{\beta}_{\alpha} \}_{\alpha \in A}$ in the set of generators $S$, then the closed subcategory $Z$ has ideal $\{ \sum_{\beta} I^{\beta}_{\alpha} \}_{\alpha \in A}$.  
\end{enumerate}
\end{proposition}

The next result shows how the intersection operation interacts with quotient categories.

\begin{proposition}
\label{prop:intersection}
Let $X$ be a Grothendieck category with localizing subcategory $Y$ and quotient category $\mc{X} = X/Y$.  
Let $\{ Z_{\beta} \}_{\beta \in B}$ be a collection of weakly closed subcategories of $X$, and let $\mc{Z}_{\beta} = \pi(Z_{\beta}) \subseteq \mc{X}$.  Recall the notation introduced before Theorem~\ref{thm:quotientweaklyclosed}.
\begin{enumerate}
\item If every $Z_{\beta}$ is $Y$-essentially stable, then so is $\bigcap_{\beta} Z_{\beta}$.
\item $\pi(\bigcap_{\beta} Z_{\beta}) = \bigcap_{\beta} \pi(Z_{\beta})$ if either (i) every $Z_{\beta}$ is $Y$-essentially stable or (ii) $B$ is a finite index set.
\item The $Y$-weakly closed subcategory of $X$ corresponding to $\bigcap_{\beta} \mc{Z}_{\beta}$ in Theorem~\ref{thm:quotientweaklyclosed} is equal to $(\bigcap_{\beta} \wh{Z}_{\beta})'$.
\end{enumerate}
\end{proposition}
\begin{proof}
(1) Since all $Z_{\beta}$ are stable under the operation $\omega \pi$, so is $\bigcap_{\beta} Z_{\beta}$.

(2) It is clear that $\pi(\bigcap_{\beta} Z_{\beta}) \subseteq \bigcap_{\beta} \pi(Z_{\beta})$.
Conversely, let $\mc{M} \in \bigcap_{\beta} \pi(Z_{\beta})$ and for each $\beta$ choose a $Y$-torsionfree object $M_{\beta} \in Z_{\beta}$ such that $\pi(M_{\beta}) = \mc{M}$.  For all $\beta$, $N = \omega(\mc{M}) = \omega \pi(M_{\beta})$ is isomorphic to the largest essential extension inside an injective hull of $M_{\beta}$ by objects of $Y$.  Thus without loss of generality we can think of all of the $M_{\beta}$ as subobjects of $N$.

If every $Z_{\beta}$ is $Y$-essentially stable, then $N \in \omega\pi(Z_{\beta}) = Z_{\beta}$ for all $\beta$, and as 
$\pi(N) = \mc{M}$ we have $\mc{M} \in \pi(\bigcap_{\beta} Z_{\beta})$.

If instead the index set $B$ is finite, define $M = \bigcap_{\beta} M_{\beta}$ as subobjects of $N$.  
Now $N/M$ embeds in $\bigoplus_{\beta} (N/M_{\beta})$, which is an object of $Y$.  Thus $\pi(M) = \pi(N) = \mc{M}$, and 
as $M \subseteq M_{\beta}$ we have $M \in Z_{\beta}$ for all $\beta$.  Again $\mc{M} \in \pi(\bigcap_{\beta} Z_{\beta})$.

(3) Recall that $\wh{Z}_{\beta}$ is the $Y$-weakly closed subcategory generated by all $Y$-torsionfree objects $M$ such that $\pi(M) \in \pi(Z_{\beta}) = \mc{Z}_{\beta}$.  Since each $\wh{Z}_{\beta}$ is $Y$-essentially stable, so is $\bigcap_{\beta} \wh{Z}_{\beta}$ by (1).  Then $(\bigcap_{\beta} \wh{Z}_{\beta})'$ is $Y$-weakly closed by Proposition~\ref{prop:operations-on-Z}(4).  We have 
\[
\pi((\bigcap_{\beta} \wh{Z}_{\beta})') = \pi(\bigcap_{\beta} \wh{Z}_{\beta}) = \bigcap_{\beta} \pi(\wh{Z}_{\beta}) = \bigcap_{\beta} \mc{Z}_{\beta}
\]
using (2).  So $(\bigcap_{\beta} \wh{Z}_{\beta})'$ must be the unique $Y$-weakly closed subcategory of $X$ corresponding to $\bigcap_{\beta} \mc{Z}_{\beta}$ in the bijective correspondence of Theorem~\ref{thm:quotientweaklyclosed}.
\end{proof}

\begin{example}
Let $X = \rGr R$ be the category of $\mb{Z}$-graded right modules $M = \bigoplus_{n \in \mb{Z}} M_n$ over the graded ring $R = \kk[x]$, where $\kk$ is a field.  Let $Y = \rTors R$ be localizing subcategory consisting of direct limits of modules which are finite-dimensional over $\kk$.  Let $\pi: X \to X/Y$ be the quotient map.  As mentioned in the introduction, in this case $X/Y = \rQgr R \simeq \Qcoh(\cProj \kk[x]) = \Qcoh(\on{Spec} \kk) \simeq \rcatMod \kk$.  Thus every weakly closed subcategory of $X/Y$ is $0$ or all of $X/Y$.

For each $m \in \mb{Z}$, let $Z_m \subseteq X$ consist of all graded modules $M$ such that $M_n = 0$ for all $n \leq m$.  It is easy to see that $Z_m$ is closed in $X$ \cite[Remark 1, p. 2139]{Smi02}.  Clearly $\bigcap_{m \geq 0} Z_m$ is the $0$-category, while $\pi(Z_m) = X/Y$ for all $m \geq 0$ because $Z_m$ contains the $Y$-torsionfree object $R(-m)$.  Thus $0 = \pi(\bigcap_{m \geq 0} Z_m) \neq \bigcap_{m \geq 0} \pi(Z_m) = X/Y$, showing that we cannot expect $\pi$ to commute with intersections in general without additional hypotheses (as in Proposition~\ref{prop:intersection}(2)).
\end{example}

\begin{example}
Let $R = \kk[x,y]$ for a field $\kk$, let $X = \rcatMod R$, and consider the closed subcategories $Z_1 = \rcatMod R/(x)$ and $Z_2 = \rcatMod R/(y)$.  Define $Y$ to be the localizing subcategory associated to the maximal ideal $\mf{m} = (x,y)$ as in Example~\ref{ex:idealpowers}.

It is clear that $Z_1$ and $Z_2$ are $Y$-torsionfree generated, while $Z_1 \cap Z_2 = \rcatMod R/\mf{m}$ (using Proposition~\ref{prop:int-gen}(2)) is not $Y$-torsionfree generated.  Moreover, $Z_1$ and $Z_2$ are $Y$-essentially stable and thus $Y$-closed in $X$, according to Example~\ref{ex:commutative}.  This shows that we need to take $(\wh{Z}_1 \cap \wh{Z}_2)'$ as the $Y$-closed subcategory corresponding to $\pi(Z_1) \cap \pi(Z_2)$ in general (as in Proposition~\ref{prop:intersection}(3)), rather than simply $\wh{Z}_1 \cap \wh{Z}_2$.
\end{example}

We move on to unions.  For any small collection of weakly closed subcategories $\{ Z_{\beta} \}_{\beta \in B}$ of $X$, we define $\bigcup_{\beta} Z_{\beta}$ to be the smallest weakly closed subcategory of $X$ containing all of the $Z_{\beta}$.  Since (weakly) closed subcategories are a formal analog of the closed subschemes of a scheme, we should not expect infinite unions to behave reasonably. So in the results below we concentrate on a union of two weakly closed subcategories, the generalization to any finite number being clear.  Given weakly closed subcategories $Z_1, Z_2$ of any Grothendieck category $X$, it is straightforward to see that  $Z_1 \cup Z_2$ can be described explicitly as the full subcategory of $X$ consisting of all subquotients of objects of the form $M \oplus N$ where $M \in Z_1$, $N \in Z_2$ \cite[Remark 2 after Definition 3.4]{Smi02}.

\begin{proposition} Let $X$ be a Grothendieck category with a small set $S = \{ O_{\alpha} \}_{\alpha \in A}$ of compact projective generators.  $Z_1$, $Z_2$ be weakly closed subcategories of $X$, with corresponding filter systems $\{ \mc{F}^1_{\alpha} \}$, $\{ \mc{F}^2_{\alpha} \}$.   Let $Z = Z_1 \cup Z_2$.
\label{prop:union-gen}
\begin{enumerate}
    \item The filter system in the generators corresponding to $Z$ is $\{ \mc{E}_{\alpha} \}$, where 
    \[\mc{E}_{\alpha} = \{ K \subseteq O_{\alpha} | \,\, I \cap J \subseteq K\ \text{for some}\ I \in \mc{F}^1_{\alpha}, J \in \mc{F}^2_{\alpha} \}.
    \]
    \item Suppose that $Z_1$ and $Z_2$ are closed in $X$, with corresponding ideals $\{ I^1_{\alpha} \}$ and $\{ I^2_{\alpha} \}$ in the generators. Then $Z$ is also closed in $X$, with ideal $\{ I^1_{\alpha} \cap I^2_{\alpha} \}$.
    \end{enumerate}
\end{proposition}
\begin{proof}
(1) Let $\{ \mc{E}_{\alpha} \}$ be the filter system corresponding to $Z$.  For any generator $O_{\alpha}$ and $I \in \mc{F}^1_{\alpha}, J \in \mc{F}^2_{\alpha}$, then $O_{\alpha}/(I \cap J)$ embeds in $(O_{\alpha}/I) \oplus (O_{\alpha}/J) \in Z_1 \cup Z_2$, so $I \cap J \in \mc{E}_{\alpha}$ and consequently any $K$ with $I \cap J \subseteq K$ is in $\mc{E}_{\alpha}$.  

Conversely, suppose that $M \in Z_1, N \in Z_2$, and $P \subseteq Q \subseteq (M \oplus N)$, so that $Q/P$ represents an arbitrary element of $Z_1 \cup Z_2$.  Consider a morphism $f: O_{\alpha} \to Q/P$, so $\ker f$ is an arbitrary element of $\mc{E}_{\alpha}$.  Let $\rho: Q \to Q/P$ be the quotient map.  Since $O_{\alpha}$ is projective we can choose a morphism $\wt{f}: O_{\alpha} \to Q$ such that $\rho \wt{f} = f$.  Clearly $\ker \wt{f} = \ker \pi_1 \wt{f} \cap \ker \pi_2 \wt{f}$, where $\pi_1, \pi_2$ are the two projection maps of $M \oplus N$, restricted to $Q$.  Now $I = \ker \pi_1 \wt{f} \in \mc{F}^1_{\alpha}$ and $J = \ker \pi_2 \wt{f} \in \mc{F}^2_{\alpha}$, and $I \cap J = \ker \wt{f} \subseteq \ker f$.

(2) This follows easily from (1).
\end{proof}

\begin{proposition}
\label{prop:union}
Let $X$ be a Grothendieck category with localizing subcategory $Y$.  Let $\pi: X \to X/Y$ be the quotient map and let $Z_1$ and $Z_2$ be weakly closed in $X$.
\begin{enumerate}
\item $\pi(Z_1) \cup \pi(Z_2) = \pi(Z_1 \cup Z_2)$.  
\item If $Z_1$ and $Z_2$ are $Y$-torsionfree generated, then so is $Z_1 \cup Z_2$.
\item Suppose that $X$ is an (AB4*) category.  If $Z_1$ and $Z_2$ are closed in $X$, then so is $Z_1 \cup Z_2$.  If in addition $Z_1 \cup Z_2$ is $Y$-essentially stable in $X$, then $\pi(Z_1) \cup \pi(Z_2)$ is closed in $X/Y$. 
\end{enumerate}
\end{proposition}
\begin{proof}
(1) Let $P \subseteq Q \subseteq (M \oplus N)$ where $M \in Z_1$ and $N \in Z_2$, so that $Q/P$ is an arbitrary object of $Z_1 \cup Z_2$.  By exactness of $\pi$, $\pi(Q/P) = \pi(Q)/\pi(P)$ is a subquotient of $\pi(M) \oplus \pi(N)$, so $\pi(Q/P)$ is in $\pi(Z_1) \cup \pi(Z_2)$ by definition.  On the other hand, suppose that $\mc{P} \subseteq \mc{Q} \subseteq \mc{M} \oplus \mc{N}$ where $\mc{M} \in \pi(Z_1)$ and $\mc{N} \in \pi(Z_2)$, so that $\mc{Q}/\mc{P}$ is an arbitrary object of $\pi(Z_1) \cup \pi(Z_2)$.  We can choose objects $M \in Z_1$ and $N \in Z_2$ so that $\pi(M) = \mc{M}$ and $\pi(N) = \mc{N}$.  Since $\pi(M \oplus N) = \mc{M} \oplus \mc{N}$, there exist subobjects $P \subseteq Q \subseteq (M \oplus N)$ such that $\pi(P) = \mc{P}$ and $\pi(Q) = \mc{Q}$ \cite[Proposition 4.10]{Kan15}.  Then $\pi(Q/P) \cong \mc{Q}/\mc{P}$ so $\mc{Q}/\mc{P} \in \pi(Z_1 \cup Z_2)$.

(2) Similarly as in (1), let $P \subseteq Q \subseteq (M \oplus N)$ where $M \in Z_1$ and $N \in Z_2$.  We can choose $Y$-torsionfree objects $M'$ and $N'$ and epimorphisms $\rho_1: M' \to M$, $\rho_2: N' \to N$, as $Z_1, Z_2$ are $Y$-torsionfree generated.  Then clearly there is  $Q' \subseteq (M' \oplus N')$ such that $\rho(Q') = Q$, where $\rho = \rho_1 \oplus \rho_2: M' \oplus N' \to M \oplus N$.  Now $Q'$ is $Y$-torsionfree, and there are epimorphisms $Q' \to Q \to Q/P$.  So every object in $Z_1 \cup Z_2$ is an epimorphic image of a $Y$-torsionfree object.

(3)  Take a small-indexed family of objects $\{ L_{\beta} \}$ such that  $L_{\beta} \cong Q_{\beta}/P_{\beta}$ with $P_{\beta} \subseteq Q_{\beta} \subseteq M_{\beta} \oplus N_{\beta}$, where $M_{\beta} \in Z_1$ and $N_{\beta} \in Z_2$.  Because products are exact in $X$, $\prod L_{\beta} \cong (\prod Q_{\beta})/(\prod P_{\beta})$ where 
$\prod P_{\beta} \subseteq \prod Q_{\beta} \subseteq \prod M_{\beta} \oplus \prod N_{\beta}$.  Since $Z_1$ and $Z_2$ are closed under products, we get that $\prod L_{\beta} \in Z_1 \cup Z_2$.  So $Z_1 \cup Z_2$ is closed.  

If $Z_1 \cup Z_2$ is $Y$-essentially stable, then $\pi(Z_1 \cup Z_2)$ is closed by Theorem~\ref{thm:quotientclosed}(1), 
and $\pi(Z_1\cup Z_2) = \pi(Z_1) \cup \pi(Z_2)$ by (1).  
\end{proof}

We have seen that intersections preserve the property of being $Y$-essentially stable, but not the $Y$-torsionfree generated property in general.  Conversely, finite unions preserve the $Y$-torsionfree generated property, but as we will see in the next example, not necessarily the $Y$-essentially stable property.  This has a more serious consequence, as it means that in a general Grothendieck category, the union of two closed subcategories may not be closed.

\begin{example}
\label{ex:badunion}
Consider the quantum polynomial ring $R = \kk\langle x_1, x_2, x_3, x_4 \rangle/(x_jx_i -p_{ij}x_ix_j \,|\, i < j)$, where we take $p_{12} = p, p_{13} = p^{-1}, p_{14} = p^{-1}, p_{23} = p_{24} = p_{34} = 1$ for some $p \neq 1 \in \kk$.  Consider the maximal ideal $I = (x_1-1, x_2, x_3, x_4)$ of $R$ and let $Y$ be the localizing subcategory of $X = \rcatMod R$ determined by $I$ as in Example~\ref{ex:idealpowers}.  Let $z_1 = x_2x_3$ and $z_2 = x_2x_4$, which commute with $x_1$ and so are central in $R$.  Let $Z_1 = \rcatMod R/(z_1)$ and $Z_2 = \rcatMod R/(z_2)$, which are closed subcategories of $X$.

For $i =1, 2$, the relation $I (z_iR) = z_iI$ shows that $z_iR \subseteq \wt{I z_i R}$ and thus 
$Z_i$ is $Y$-essentially stable in $X$, with $\pi(Z_i)$ closed in $X/Y$, according to the analysis in Example~\ref{ex:idealpowers}.  We know that $Z_1 \cup Z_2$ is closed in $X$ and defined by the ideal $z_1 R \cap z_2 R = x_2x_3 R \cap x_2x_4 R = x_2 x_3 x_4 R$,  by Proposition~\ref{prop:union-gen}.   The element $z_3 = x_2x_3 x_4$ is normal in $R$ but not central.  Specifically, let $\psi: R \to R$ be the automorphism given by 
$z_3 r = \psi(r) z_3$; then $\psi(x_1) = p^{-1}x_1$ and thus $\psi(I) = (p^{-1}x_1 -1, x_2, x_3, x_4)$ is comaximal with $I$.  Now Example~\ref{ex:pinotclosed} shows that $\pi(Z_1 \cup Z_2)$ is not closed in $X/Y$.  In particular, $Z_1 \cup Z_2$ is not $Y$-essentially stable.    On the other hand, $\pi(Z_1) \cup \pi(Z_2) = \pi(Z_1 \cup Z_2)$ by Proposition~\ref{prop:union}(1).  So the union of the two closed subcategories $\pi(Z_1)$ and $\pi(Z_2)$ of $X/Y$ is not closed in $X/Y$, only weakly closed.
\end{example}

Notice that the collection of all weakly closed subcategories of a Grothendieck category $X$ is a set.  This is because if $O$ is a generator of $X$, then any weakly closed subcategory $Z$ is determined by the set of subobjects $\mc{I}$ of $O$ such that $O/\mc{I} \in Z$.
Smith has asked if $Z_1 \cap (Z_2 \cup Z_3) = (Z_1 \cap Z_2) \cup (Z_1 \cap Z_3)$ for all weakly closed subcategories $Z_1, Z_2, Z_3$ in any Grothendieck category $X$ \cite[Remark 5, p. 2139]{Smi02}.   The weakly closed subcategories of $X$ form a lattice under $\vee = \cup$ and $\wedge = \cap$, and so an alternative way to phrase Smith's question is to ask whether that lattice is distributive.  In fact this is not true even in the setting of commutative affine schemes.
\begin{example}
\label{ex:notdist}
Let $R$ be a commutative ring and let $X = \rcatMod R$.  Let $Z_i = \rcatMod R/I_i$ for some ideals $I_i$ with $i = 1, 2, 3$, so the $Z_i$ are three arbitrary closed subcategories of $X$.  
Note that the collection $\on{Ideals}(R)$ of all ideals of $R$ is a lattice with $I \vee J = I + J$ and $I \wedge J = I \cap J$.  A ring $R$ is called \emph{arithmetical} if $\on{Ideals}(R)$ is a distributive lattice, and this is known to be a very stringent condition.  For example, a noetherian integral domain $R$ is arithmetical if and only if it is Dedekind.   

Thus there is a counterexample to Smith's conjecture already for $R = \kk[x,y]$, where $\kk$ is a field.  Explicitly, we can take $I_1 = (x)$, $I_2 = (y)$, $I_3 = (x+y)$.  Note that $(x) \cap ((y) + (x+y)) = (x) \cap (x,y) = (x)$ while $((x) \cap (y)) + ((x) \cap (x+y)) = (xy) + (x(x+y)) = (x)(x, y)$, so $\on{Ideals}(R)$ is not distributive. By Propositions~\ref{prop:int-gen}(2) and \ref{prop:union-gen}(2), the closed subcategories of $X$ form a lattice under $\vee = \cup$ and $\wedge = \cap$ which is dual to $\on{Ideals}(R)$.  Since distributivity is a self-dual property of a lattice, the lattice of closed subcategories of $X$ is not distributive, and therefore neither is the larger lattice of weakly closed subcategories.  
\end{example}

Of course the closed \emph{subsets} of a scheme under the Zariski topology do form a distributive lattice under intersection and union, as in any topological space.  Kanda has defined the notion of a reduced Grothendieck category \cite{Kan22}.  Perhaps one may define a notion of union and intersection for reduced weakly closed subcategories, such that distributivity holds.

To close the paper, we give a more detailed comparison of the results and notation of our paper with those of Smith in \cite{Smi02}.  In this paper we have treated the quotient category $X/Y$ as an algebraic object and sought to find a ``first isomorphism theorem" kind of relationship between the closed subcategories of $X$ and those of a quotient category $X/Y$.  

By contrast, Smith interprets the quotient categories $X/Y$ geometrically, as suggested by his word ``space" for a Grothendieck category.  Given a space $X$ with weakly closed subspace $Z$, Let $Y$ be the closure of $Z$ under extensions, which is a localizing subcategory.  Smith defines the \emph{complement} to $Z$ to be the space $X/Y$, which is called  \emph{weakly open} and denoted $X \backslash Z$.  More precisely, $X/Y$ is treated as a subcategory of $X$ via its image under the full and faithful functor $\omega: X/Y \to X$.  Of course this image $\omega(X/Y)$ is not generally an abelian subcategory of $X$.  Weakly open spaces can also be characterized functorially \cite[Definition 2.5]{Smi02}.  

The ``almost topology" of (weakly) open and closed subspaces, and their intersections and unions, is studied extensively in \cite[Section 6-7]{Smi02}.  Smith gives examples which show that the geometric intuition given by idea of open complement can be misleading.  These examples are generally related to the failure of the $Y$-essentially stable property, in our language.
For instance, Example~\ref{ex:noncommutative} is also given by Smith \cite[Example 6.11]{Smi02},
to show that for the weakly closed subspace $Z = Z_1 \cup Z_2$ and localizing subspace $Y = Z_1$, $(X \backslash Z_1) \cap Z = \emptyset$; that is, the complement of one point in a two-point space may not contain the other point.  Here, $X \backslash Z_1 = X/Y$ is embedded in $X$ as all direct sums of the non-simple indecomposable module $P_2$, which does not have any nonzero objects in common with the semisimple category $Z$.  It is exactly the 
fact that $Z$ is not $Y$-essentially stable that allows this to happen.

Although we did not make the connection until after our results were proved, Smith's \cite[Definition 7.8]{Smi02} and the results following it even anticipate the form of the bijective correspondence we give in Theorem~\ref{thm:quotientweaklyclosed} between weakly closed subcategories of $X/Y$ and $Y$-weakly closed subcategories of $X$.  In more detail, suppose that $W$ is weakly closed in $X$ and let $Y$ be the localizing subcategory it generates, so that explicitly 
$U = X \backslash W = \omega(X/Y) \subseteq X$ is weakly open.  Let $Z$ be weakly closed in $X$. Smith defines $U \cap Z$ to be the objectwise intersection of these subcategories of $X$ \cite[Definition 7.8]{Smi02}, but only under an additional condition that is equivalent to $Z$ being $Y$-essentially closed in our language.  He shows that every weakly closed subspace of $V$ of $U$ has the form $\overline{V} \cap U$ \cite[Proposition 7.10]{Smi02}, where $\overline{V}$ is a weakly closed subcategory of $X$ called the \emph{weak closure} of $V$. In our setup, given a weakly closed subcategory $\mc{V}$ of $X/Y$, Smith's weak closure can be seen to equal $(\pi^{-1}(\mc{V}))'$, that is, the $Y$-weakly closed subcategory of $X$ corresponding to $\mc{V}$ in Theorem~\ref{thm:quotientweaklyclosed}.

\end{document}